\documentclass[11pt]{article}
\usepackage[left=1in,right=1in,bottom=1in,top=1in]{geometry}
\usepackage{amsmath}
\usepackage{amsfonts}
\usepackage{amssymb}
\usepackage{amsthm}
\usepackage{young}
\usepackage[vcentermath]{youngtab}
\usepackage{tikz}
\usepackage{enumitem}

\usepackage{algorithm}
\usepackage[noend]{algpseudocode}
\usepackage{amsmath, amssymb, amsthm}

\makeatletter
\newtheorem*{rep@theorem}{\rep@title}
\newcommand{\newreptheorem}[2]{%
\newenvironment{rep#1}[1]{%
 \def\rep@title{#2 \ref{##1}}%
 \begin{rep@theorem}}%
 {\end{rep@theorem}}}
\makeatother
\usepackage{graphicx}

\newtheorem{theorem}{Theorem}
\newreptheorem{theorem}{Theorem}
\newtheorem{lemma}[theorem]{Lemma}

\newtheorem{remark}[theorem]{Remark}
\newtheorem{definition}[theorem]{Definition}
\newtheorem{proposition}[theorem]{Proposition}
\newtheorem{corollary}[theorem]{Corollary}

\DeclareMathOperator{\id}{id}

\usepackage{tikz}
\usetikzlibrary{calc}

\title{Cutoff for the cyclic adjacent transposition shuffle}
\author{Danny Nam\thanks{Department of Mathematics, Princeton University, USA.} \and Evita Nestoridi\footnotemark[1]}
\date{}
\begin{document}
\maketitle
\begin{abstract}
We study the cyclic adjacent transposition (CAT) shuffle of $n$ cards, which is a systematic scan version of the random adjacent transposition (AT) card  shuffle. 
In this paper, we prove that the CAT shuffle exhibits cutoff at $\frac{n^3}{2 \pi^2} \log n$, which concludes that it is twice as fast as the AT shuffle.
\end{abstract}

\section{Introduction}
How long does it take to shuffle a deck of cards sufficiently well? Mixing time of card shuffling schemes and Markov chains in general is a widely studied subject in probability. Recently, there has been a lot of interest in understanding the behavior of time-inhomogeneous chains and in sharpening the techniques that have been developed in the time-homogeneous case (see \cite{BCSV, MNP, MPS, Pinsky, SZ1, SZ2, SZ3, SZ4}). In the present paper, we study the mixing time of the cyclic adjacent transposition shuffle and show that it exhibits cutoff, which is the first verification of cutoff phenomenon for a time-inhomogeneous card shuffle.


The cyclic adjacent transposition (CAT) shuffle is a systematic scan version of the adjacent transposition shuffle. In the CAT shuffle, we start with a deck of $n$ cards, that are placed on the vertices on a path of length $(n-1)$. At the beginning of the first step, we flip a fair coin, which determines if we are going to move from left to right or from right to left. If we do the former, then at time $t=1$ with probability $1/2$ we transpose the cards at the ends of the first edge, otherwise we stay fixed. For $t=2,\ldots n-1$, with probability $1/2$ we transpose cards at the ends of the $t$-th edge, otherwise we stay fixed, etc.  If we move from right to left, at time $t=1,\ldots n-1$, with probability $1/2$ we transpose the cards that lie on the ends of the $(n-t)$-th edge, otherwise we do nothing. 

In other words, we explore the deck from the first card to the last card with respect to the direction we choose at the beginning, and independently at each step either swap the positions of the neighboring ones or stay fixed according to a fair coin toss. When $t \equiv 1 \bmod (n-1)$, we repeat the first $(n-1)$ steps of the chain independently, i.e., pick the orientation (either from $1$ to $n$ or from $n$ to $1$) uniformly at random, move from the first card to the last one according to the chosen direction, and at each step either transpose or do nothing uniformly independently at random. 

The configuration space of the CAT shuffle is the symmetric group $S_n$. Let $x,y \in S_n$ and let $P^t_{x}(y)$ be the probability of moving from the $x$ to $y$ in $t$ steps. Then the basic limit theorem of Markov chains tells us that $P^t_{x}$ converges to the uniform measure $\mu$ as $t \rightarrow \infty$ with respect to the total variation distance
\begin{equation*}
d_x(t):=\|P^t_{x}-\mu \|_{T.V.} := \frac{1}{2} \sum_{y \in S_n}|P^t_{x}(y) - \mu(y)|.
\end{equation*}
The mixing time of this Markov chain is defined as
\begin{equation*}
t_{mix}(\varepsilon)= \min \{t \in \mathbb{N}: \max_{x \in S_n}\{ d_x(t) \}\leq \varepsilon\}.
\end{equation*}
Our main result provides sharp bounds for the mixing time of the CAT shuffle.
\vspace{1mm}

\begin{theorem}\label{main}
For the cyclic adjacent transposition shuffle, we have that for any $\varepsilon >0$,
\begin{enumerate}
\item[\textnormal{(a)}] There is a universal constant $c$, such that
$t_{mix} (1- \varepsilon) \geq \frac{n^3}{2 \pi^2} \log n - \frac{n^3}{2 \pi^2} \log \left(\frac{c \log n}{\varepsilon} \right).$
\item[\textnormal{(b)}] $t_{mix}(\varepsilon) \leq (1+o(1)) \frac{n^3}{2\pi^2} \log n.$
\end{enumerate}
\end{theorem}

Theorem \ref{main} says that the cyclic adjacent transposition shuffle exhibits cutoff at $\frac{n^3}{2 \pi^2} \log n,$ i.e. that there is window $w_n= o(n^3 \log n)$ such that
\begin{equation*}
\lim_{k \rightarrow \infty} \lim_{n \rightarrow \infty} d \left( \frac{n^3}{2 \pi^2} \log n - kw_n\right)= 1 ~~~\mbox{ and }	~~~ \lim_{k \rightarrow \infty} \lim_{n \rightarrow \infty} d \left( \frac{n^3}{2\pi^2} \log n +kw_n\right)= 0.
\end{equation*}

\vspace{1mm}

As mentioned above, the CAT shuffle is a systematic scan version of the adjacent transposition (AT) shuffle. In the AT shuffle, with probability $1/2$ we transpose a random adjacent pair of cards and otherwise do nothing. It is an important card shuffling model mainly because of its connection to the exclusion process. Only recently, Lacoin \cite{Lacoin} proved the sharp upper bound for the mixing time of the AT shuffle, which combined with the sharp lower bound of Wilson \cite{Wilson} concluded the proof of cutoff for this model. They also established the same result for the simple exclusion process, verifying the close connections between the two models.

The first time-inhomogeneous card shuffle to be studied is the semi-random transposition card shuffle, which suggests that at time $t$ we transpose the card in position $t \bmod n$ with a uniformly random card. It was introduced by Thorp \cite{Thorp}, and Aldous and Diaconis \cite{A-D} first raised the question of determining the mixing time of semi-random transpositions. Mironov \cite{Mironov} used this model for a cryptographic system and proved that the mixing time is at most $O(n \log n)$. Mossel, Peres and Sinclair \cite{MPS} established a matching lower bound of order $\Theta(n \log n)$. This lower bound was obtained using Wilson's method \cite{Wilson}, which relies on finding an appropriate eigenfunction.

Another well-studied time-inhomogeneous card shuffle  is the card-cyclic-to-random shuffle. In this model, at time $t$ we remove the card with the label $t \bmod n$ and insert it to a uniformly random position of the deck. This model was introduced by Pinsky \cite{Pinsky}, who showed that $n$ steps are not sufficient to shuffle the deck well enough. Morris, Ning, Peres \cite{MNP} later proved both a lower and  an upper bound of order $n \log n$.

Saloff-Coste and Zuniga \cite{SZ1, SZ2, SZ3, SZ4} studied time-inhomogeneous Markov chains via singular value decomposition. In their work, they find better constants for the upper bound for both semi-random transpositions and card-cyclic-to-random shuffles. Their result is based on bounding the singular values of the transition matrix of the time-inhomogeneous chains by the eigenvalues of the time-homogeneous card shuffles. Although very useful in some models, their technique  does not work well enough in our case.

Very recently, Angel and Holroyd \cite{AH} asked a different question concerning a similar model; given a sequence of parameters $S = (a_i, b_i, p_i)_{i=1}^{\ell}$, at time $t=1, \ldots, \ell$ with probability $p_t$ they transpose card $a_t$ with the card $b_t$, otherwise do nothing. They study the question of finding the minimum length $\ell$ such that the resulting permutation of $n$ cards is random. They prove that the for the case that $b_i=a_i+1$, this minimum length is exactly ${n \choose 2}$.

Another model one can consider is the \textit{single-directional CAT shuffle}, which at time $t$ swaps the cards at positions $t$, $t+1$ mod $n-1$. In other words, it is a variant of the CAT shuffle that explores the deck in a single direction rather than renewing it at every $n-1$ steps. In this model, we have the same upper bound on the mixing time as part (b) of Theorem 1, and indeed the proof  works analogously for this case. However, the techniques used to prove part (a) no longer applies to this model due to lack of symmetry. In the CAT shuffle, setting a random direction of exploartion at every $n-1$ steps provides some amount of symmetry which makes it more convenient to carry out our approach. We conjecture that the single-directional CAT shuffle exhibits cutoff at $\frac{n^3}{2\pi^2}\log n$, the same location as the CAT shuffle.

\subsection{Proof outline}

The main difficulty of studying the CAT shuffle comes from its deterministic selections of update locations. Due to this aspect, it seems impossible to write down the closed formula of the transition using eigenvalues and eigenfunctions, although most of the properties of the AT shuffle can be deduced by this approach \cite{Wilson, Lacoin}. To overcome this difficulty, we rely on the following observations:
\begin{enumerate}
	\item [(\romannumeral 1)] We can compute ``approximate eigenfunctions'', which behave like the actual eigenfunctions but with errors.
	
	\item [(\romannumeral 2)] When $n$ is large enough, each card follows a Brownian-type move under an appropriate scaling of $n$ and $t$.
\end{enumerate}

To prove the lower bound on the mixing time, we derive a generalized version of  Wilson's lemma \cite{Wilson} which enables to implement the ``approximate eigenfunctions'' obtained from  observation (\romannumeral 1). Using this lemma, we conclude the first part of Theorem  \ref{main} by showing that the errors of the approximate eigenfunctions are small enough.

For the upper bound, we rely on the idea of monotone coupling and censoring from Lacoin \cite{Lacoin}; by defining the ``height'' of card decks, we can construct a monotone coupling of the system and take advantage of the censoring inequality following the approach of \cite{Lacoin}. 

In this procedure, a crucial ingredient we need is that the height of the deck decays exponentially in time according to the correct rate. In the AT shuffle \cite{Lacoin}, this property is derived based on the algebraic relations of the model. Since this approach seems impossible for the CAT shuffle, we take account of (\romannumeral 2) to deduce such condition.

\subsection{Organization of the paper}
In \S \ref{Wilson's}, we derive a generalized Wilson's lemma that works for approximate eigenfunctions. Then in \S \ref{subsecmoment}, we introduce the appropriate approximator to study and show that the error is small enough to deduce the correct lower bound. Based on this result, we conclude the proof of part (a) of Theorem \ref{main} in \S \ref{subsecmain1pf}.

Section \ref{seconec} is devoted to understanding the movement of a single card. Here, we explain the precise meaning of  observation (\romannumeral 2) above and deduce hitting time estimates of a single card. The monotone coupling, the censoring inequality and the exponential decay of the ``height'' are explained in \S \ref{subsecmono}, and we prove part (b) of Theorem \ref{main} in \S \ref{subsecpfmain2}. 

In the final section \S \ref{secex},  as an application of our main theorem we study the systematic simple exclusion process which is the particle system version of the CAT shuffle. 
 
\subsection*{Acknowledgements}
We are grateful to Yuval Peres and Allan Sly for their helpful comments.
\section{Generalizing Wilson's lemma}\label{Wilson's}
For the lower bound, we will need a generalization of Wilson's lemma \cite{Wilson}. The main difference is that we do not use the precise eigenfunctions of the transition matrix $P$, but rather functions that behave sufficiently like eigenfunctions.
\begin{lemma}\label{W}
Let $X_t$ be a Markov chain on a state space $\Omega_n$, with stationary distribution $\mu$. Let $x_0\in \Omega$. Suppose that there are parameters $\gamma, \delta, R>0$ and a function $\Phi :\Omega_n \rightarrow \mathbb{R}$ such that $\Phi(x_0)>0$, satisfying the following:
\begin{enumerate}
\item[\textnormal{(a)}] The mean of $\Phi$ under stationarity is zero, that is $\mu(\Phi)=0$. 
\item[\textnormal{(b)}] We have $ 0< \gamma < 2-\sqrt{2}$ and for all $t \geq 0 $ it holds that
$$|\mathbb{E} [\Phi(X_{t+1}) \vert X_{t}] - (1-  \gamma)\Phi( X_{t})| \leq \delta.$$  
\item[\textnormal{(c)}] We also have that  $\mathbb{E}[(\Delta \Phi_t)^2 \vert X_{t}] \leq R$,
where $\Delta \Phi_t := \Phi(X_{t+1})- \Phi(X_t)$.
\end{enumerate}
Then for $t=\frac{1}{ \gamma_\star}\log (\Phi(x_0))- \frac{1}{2\gamma_\star}\log \left( \frac{48(\delta \Vert\Phi\Vert_\infty +R)}{\gamma \varepsilon} \right)$, we have
$$\Vert P^t_{x_0} -\mu\Vert_{T.V.} \geq 1-\varepsilon,$$
where $\gamma_\star := -\log (1-\gamma)$.
\end{lemma}

\begin{proof}
Let $\varepsilon >0$. By iterating the condition (b), we get that
\begin{equation}\label{ex}
\mathbb{E}_{x_0}[ \Phi(X_t)]  \geq (1- \gamma )^t\Phi(x_0) -\delta/\gamma .
 \end{equation}

\noindent To control the variance, we notice the inequality that
\begin{equation}\label{2ndmeq1}
	\begin{split}
	\mathbb{E} [(\Phi(X_{t+1}))^2 | X_t] 
	&=
	(\Phi(X_t))^2 + 2\Phi(X_t) \mathbb{E} [ \Delta \Phi_t |X_t] +
	\mathbb{E} [ (\Delta \Phi_t)^2 | X_t]\\
	&\leq
	(1-2\gamma)\Phi(X_t)^2 + (\delta \Vert \Phi \Vert_\infty + R).
		\end{split}
	\end{equation}
Iterating \eqref{2ndmeq1}, we have that 	
\begin{equation}\label{2nd1}
	\begin{split}
	\mathbb{E}_{x_0} [(\Phi(X_{t}))^2 ] 
	&\leq
	(1-2\gamma)^t\Phi(x_0)^2 + \frac{\delta \Vert \Phi \Vert_\infty+R}{2\gamma}.
	\end{split}
	\end{equation}
	
\noindent Using \eqref{ex}, this implies that 
\begin{equation}\label{var}
\begin{split}
\textnormal{Var}(\Phi(X_t) \vert X_0=x_0) \leq  
\frac{\delta \Vert \Phi \Vert_\infty +R}{2\gamma}
+\frac{2\delta \Vert \Phi \Vert_\infty }{ \gamma } ~\leq~   
\frac{3(\delta \Vert \Phi \Vert_\infty +R)}{\gamma}.
\end{split}
\end{equation}
Letting $t$ go to infinity, we also get the same bound for $\textnormal{Var}(\Phi )$ under the stationary distribution.

Let $t=\frac{1}{ \gamma_\star}\log (\Phi(x_0))- \frac{1}{2\gamma_\star}\log \left( \frac{48(\delta \Vert\Phi\Vert_\infty +R)}{\gamma \varepsilon} \right)$  and consider the event  
$$A=\left\{ x \in \Omega_n :  \Phi(x) <  \frac{1}{2}\mathbb{E}_{x_0} [\Phi(X_t)] \right\}.$$  
Then by Chebychev's inequality combined with \eqref{ex} and \eqref{var}, we have that 
\begin{equation}\label{tr}
\mathbb{P}_{x_0}\left( X_t \in A \right) \leq 
\frac{12(\delta \Vert \Phi \Vert_\infty +R) / \gamma}{(1-\gamma)^{2t} \Phi(x_0)^2 - 2\delta \Vert \Phi \Vert_\infty/\gamma } \leq \frac{\varepsilon}{2}.
\end{equation}
Similarly with respect to the stationary measure, we obtain that
\begin{equation}\label{un}
\mathbb{P}_\mu (X \in A ) \geq 1-\frac{\varepsilon}{2}.
\end{equation}
Combining \eqref{tr} and \eqref{un}, we deduce that
 
$$\Vert P^t_{x_0} -\mu\Vert_{T.V.}\geq | P^t_{x_0}(A) -\mu(A)| \geq 1-\varepsilon.$$
\end{proof}

\section{The lower bound}\label{seclowerbd}

In \cite{Wilson}, the lower bound on the mixing time for the random AT shuffle is obtained by analyzing the \textit{height function representation} of the chain. In this case, one can compute the exact eigenvalues and eigenfunctions of the transition of height functions.

On the other hand, the main difficulty of investigating the CAT shuffle is that we cannot precisely calculate such eigenvalues and eigenfunctions since the update locations are not given randomly. However, we can still overcome this obstacle by using the objects which approximately behave like eigenfunctions with small enough errors. 

In \S \ref{subsecmoment}, we introduce the \textit{height function representation} of the CAT shuffle and describe its first and the second moment estimates, based on the aforementioned idea of ``approximate eigenfunctions.'' Then, \S \ref{subsecmain1pf} is devoted to proving Theorem \ref{main}, part (a) using the ingredients obtained in subsection \ref{subsecmoment} and Lemma \ref{W}.

\subsection{The moment estimates}\label{subsecmoment}

Let $\sigma_0:= \id \in S_n$ be the starting state of the CAT shuffle and let $\sigma_s$ denote the deck at time $s$. For each $t\in \mathbb{N}$, the \textit{height function} $h_t: [n]\rightarrow \mathbb{R}$ of $(\sigma_s)$ is defined as
\begin{equation}\label{heightdef}
h_t (x) := \sum_{z=1}^x \textbf{1}{\{\sigma_{(n-1)t}(z) \leq \lfloor n/2\rfloor  \} } - \frac{ \lfloor n/2\rfloor}{n}x.
\end{equation}
 Let $\mathcal{F}_t$ denote the sigma-algebra for the shuffling until time $(n-1)t$. Our goal in this subsection is to obtain the first and the second moment estimates on the following quantity $\Phi_t$:
\begin{equation}\label{Phidef}
\Phi_t := \sum_{x=1}^{n-1} h_t(x) \sin \left(\frac{\pi x}{n} \right).
\end{equation}

We begin with the first moment estimate of $\Phi_t$. The following lemma is proven similarly as Lemma \ref{tildebdlem}, and the proof can be found in \S \ref{subsec1stm}.

\begin{lemma}\label{1stmlem}
	Let $\Phi_t,$ $\mathcal{F}_t$ defined as above. For any $t \in \mathbb{N}$ we have
	\begin{equation}\label{1stmeq1}
	|\, \mathbb{E} [\Phi_{t+1} | \mathcal{F}_t] - (1-\gamma) \Phi_t\,| \leq  \frac{3\pi}{4n} ,
	\end{equation}
	where $\gamma := \pi^2/n^2 - O(n^{-4}).$
\end{lemma}

\begin{remark}
	\textnormal{Although we cannot have a more precise form such as $ \mathbb{E} [\Phi_{t+1} | \mathcal{F}_t] = (1-\gamma) \Phi_t$ as \cite{Wilson}, Lemma \ref{Wilson's} says that the estimate of Lemma \ref{1stmlem} is sufficient to get a lower bound.}
\end{remark}

 Our next goal  is to bound the second moment of $\Phi_t$. One convenient way of doing this is to look at $\Delta \Phi_t := \Phi_{t+1}- \Phi_t$,  similar to what is done in \cite{Wilson}.

\begin{lemma}\label{2ndmlem}
 There exists an absolute constant $C>0$ such that for any $t\in \mathbb{N},$
	\begin{equation*}
	\mathbb{E} [(\Delta \Phi_t)^2 | \mathcal{F}_t] \leq Cn \log n.
	\end{equation*}
\end{lemma}

%
%

\noindent \textit{Proof of Lemma \ref{2ndmlem}.} ~ For each $a\in[n]$, let $q_t(a)$ denote the position of the card $a$ at time $(n-1)t$, i.e., $q_t(a) := \sigma_{(n-1)t}^{-1}(a)$. Observe that we can write $h_t(x)$ in terms of $q_t(a)$ in the following way:
\begin{equation}\label{hrep2}
h_t(x) = \sum_{a=1}^{\lfloor n/2\rfloor} \textbf{1}_{\{q_t(a) \leq x \}} - \frac{x}{n} \left\lfloor \frac{n}{2} \right\rfloor.
\end{equation}
Therefore, $\Delta \Phi_t = \Phi_{t+1} - \Phi_t$ becomes
\begin{equation*}
\Delta\Phi_t 
= \sum_{a=1}^{\lfloor n/2 \rfloor } \left\{\sum_{x=1}^{n-1} \left(\textbf{1}_{\{q_{t+1}(a) \leq x \}} -\textbf{1}_{\{q_t(a)\leq x \}} \right) \sin \left( \frac{\pi x}{n} \right)  \right\}
= \sum_{a=1}^{\lfloor n/2 \rfloor} \psi_t(a),
\end{equation*}
where we define $\psi_t(a)$ by
\begin{equation*}
\psi_t(a):= \sum_{x=q_{t+1}(a)}^{n-1} \sin \left(\frac{\pi x}{n} \right) - \sum_{x=q_t(a)}^{n-1} \sin \left(\frac{\pi x}{n} \right).
\end{equation*}

We begin with estimating $\mathbb{E} [\psi_t(a)^2 \, |\,\mathcal{F}_t]$. Let $\overset{\rightarrow}{\mathbb{E}}$ (resp. $\overset{\leftarrow}{\mathbb{E}}$) denote the conditional expectation given the event that we explore the deck from position $1$ to $n$ (resp. $n$ to $1$) over the time period of $(n-1)t+1$ to $(n-1)(t+1)$. In other words, if $\textbf{c}_t \in \{1,n\} $ is the random variable that denotes the starting position of exploration at time $(n-1)t$, then $\overset{\rightarrow}{\mathbb{E}}[\:\cdot\:|\mathcal{F}_t ] =\mathbb{E}[\:\cdot\:|\mathcal{F}_t,\:\textbf{c}_t = 1]  $. Recall that $q_{t+1}(a)-q_{t}(a)$ follows the distribution (\ref{jumpprob1}, \ref{jumpprob2}). Letting $j$ count the displacement of card $a$, we have that for $2\leq q_t(a) \leq  n-1$, 
\begin{equation}\label{psi2eq0}
\begin{split}
\overset{\rightarrow}{\mathbb{E}} [\psi_t(a)^2\, | \,\mathcal{F}_t]  
&\leq
\frac{1}{2} \sin^2 \left(\frac{\pi(q_t(a)-1)}{n} \right)
+ \sum_{k=1}^\infty \frac{1}{2^{k+2}} 
\left\{\sum_{j=0}^{k-1} \sin \left(\frac{\pi (q_t(a) + j)}{n} \right)  \right\}^2 \\
&\leq \frac{\pi^2}{2n^2} \left\{(q_t(a)-1)^2+ \sum_{k=1}^\infty \frac{1}{2^{k+1}} \left(kq_t(a) + \frac{k(k-1)}{2} \right)^2 \right\} ~\leq~ C_1, 
\end{split}
\end{equation}
for some absolute constant $C_1 >0$, using the fact that $\sin \theta \leq \theta$ and $q_t(a)\leq n$. We can conduct a similar calculation for the cases $q_t(a)=1, n$ as well as for $\overset{\leftarrow}{\mathbb{E}} [\psi_t(a)^2\,|\,\mathcal{F}_t]$ and obtain that for all $a$,
\begin{equation}\label{psi2eq}
\mathbb{E} [\psi_t(a)^2\,|\,\mathcal{F}_t] \leq C_1.
\end{equation}

We turn our attention to estimating the correlation  and show that $ |\,\mathbb{E} [\psi_t(a)\psi_t(b)\,|\,\mathcal{F}_t]\,|
=O(
\frac{1}{n})$ for $a,\:b$ which are far apart from each other. In particular, let us assume that both $q_t(a) \geq 2$ and $ q_t(a)+ 4\log n \leq q_t(b) \leq n-1$ hold true. Define $A$ to be the event that
\begin{equation*}
A:= \{q_{t+1} (a)-q_t(a) \leq 4\log n  -2 \}.
\end{equation*}
 Then, $q_{t+1}(a)$ and $q_{t+1}(b)$ are conditionally independent given $\mathcal{F}_t$ and the event
\begin{equation*}
\{\textbf{c}_t =1 \} \cap A .
\end{equation*}
Therefore, we can express $\overset{\rightarrow}{\mathbb{E}} [\psi_t(a)\psi_t(b)\,|\,\mathcal{F}_t]$ by
\begin{equation}\label{psicoreq1}
\begin{split}
\overset{\rightarrow}{\mathbb{E}} [\psi_t(a)\psi_t(b)\,|\,\mathcal{F}_t]
&=
\overset{\rightarrow}{\mathbb{P}}(A)\;
\overset{\rightarrow}{\mathbb{E}} [\psi_t(a)\,|\, A, \mathcal{F}_t]\;
\overset{\rightarrow}{\mathbb{E}} [\psi_t(b)\,|\, A, \mathcal{F}_t]\\
&~~+
\overset{\rightarrow}{\mathbb{E}} [\psi_t(a) \psi_t(b)
\textbf{1}_{A^c} \,|\,  \mathcal{F}_t].
\end{split}
\end{equation}
Since $\mathbb{P}(A^c)\leq n^{-4}$, H$\ddot{\textnormal{o}}$lder's inequality implies that
\begin{equation}\label{psicoreq2}
\overset{\rightarrow}{\mathbb{E}} [\psi_t(a) \psi_t(b)
\textbf{1}_{A^c} \,|\,  \mathcal{F}_t]
~\leq~
\overset{\rightarrow}{\mathbb{E}} [\psi_t(a)^4\, | \,\mathcal{F}_t]^{\frac{1}{4}}\;
\overset{\rightarrow}{\mathbb{E}} [\psi_t(b)^4\, | \,\mathcal{F}_t]^{\frac{1}{4}}  \;
\overset{\rightarrow}{\mathbb{P}} (A^c)^{\frac{1}{2}}
~\leq~ \frac{C_2}{n^2},
\end{equation}
by noting that the fourth moment of $\psi_t(a)$ conditioned on $\mathcal{F}_t$ can be estimated in the same way as (\ref{psi2eq0}). On the other hand, we have
\begin{equation*}
\begin{split}
\left|\overset{\rightarrow}{\mathbb{E}}
[\psi_t(a) \textbf{1}_A\,|\,\mathcal{F}_t ] \right|
&=
\left| \frac{1}{2}\sin\left(\frac{\pi (q_t(a)-1)}{n} \right)
- \sum_{k=1}^{\lfloor 4\log n \rfloor -2} 
\frac{1}{2^{k+2}} \sum_{j=0}^{k-1} \sin \left(\frac{\pi (q_t(a)+j)}{n} \right)
 \right|\\
 &\leq
 \left| \frac{1}{2}\sin\left(\frac{\pi (q_t(a)-1)}{n} \right)
 - \sum_{k=1}^{\infty} 
 \frac{1}{2^{k+2}} \sum_{j=0}^{k-1} \sin \left(\frac{\pi (q_t(a)+j)}{n} \right)
 \right| ~+~ \frac{1}{n^3}.
\end{split}
\end{equation*}
Using $|\sin (x+\delta)-\sin(x) | \leq \delta$ to control the r.h.s., we obtain that
\begin{equation}\label{psicoreq3}
\left|\overset{\rightarrow}{\mathbb{E}}
[\psi_t(a) \textbf{1}_A\,|\,\mathcal{F}_t ] \right|
\leq
\frac{\pi}{2n} ~+~ \sum_{k=1}^\infty \frac{1}{2^{k+2}} \sum_{j=0}^{k-1} \frac{j\pi}{n} ~+~ \frac{1}{n^3} \leq \frac{C_3'}{n},
\end{equation}
for an absolute constant $C_3' >0$. Similar computations can be done for $\overset{\leftarrow}{\mathbb{E}}$. Since $\overset{\rightarrow}{\mathbb{P}}(A) \geq 1-n^{-4}$ and
$|\overset{\rightarrow}{\mathbb{E}} [\psi_t(b)\,|\,  \mathcal{F}_t]| \leq C_1$,  we deduce by combining (\ref{psicoreq1}--\ref{psicoreq3}) that
\begin{equation}\label{psicoreq}
|\,\mathbb{E} [\psi_t(a)\psi_t(b)\,|\,\mathcal{F}_t]\,|
\leq
\frac{C_3}{n},
\end{equation}
for some absolute constant $C_3 >0$.

 Let $Q \subset [\:\lfloor n/2\rfloor\: ]^2$ be defined as
\begin{equation*}
Q:=\{(a,b)\in [\:\lfloor n/2\rfloor \:]^2 : 2\leq q_t(a), q_t(b) \leq n-1, \; |q_t(a)-q_t(b)|\leq 4\log n  \}.
\end{equation*}
We also denote $Q^c:= [\,\lfloor n/2\rfloor\, ]^2 \setminus Q $. Then we can estimate $\mathbb{E}[(\Delta\Phi_t)^2|\mathcal{F}_t]$ using the inequalities (\ref{psi2eq}) and (\ref{psicoreq}) as follows.
\begin{equation*}
\begin{split}
\mathbb{E}[(\Delta\Phi_t)^2|\mathcal{F}_t]
&=
\sum_{a,b=1}^{\lfloor n/2 \rfloor} \mathbb{E} [\psi_t(a)\psi_t(b) |\mathcal{F}_t]\\
&\leq
\sum_{(a,b)\in Q^c }  \mathbb{E} [\psi_t(a)\psi_t(b) |\mathcal{F}_t]
~+
\sum_{(a,b)\in Q }  \mathbb{E} [\psi_t(a)^2 |\mathcal{F}_t]^{\frac{1}{2}} \, \mathbb{E}[ \psi_t(b)^2 |\mathcal{F}_t]^{\frac{1}{2}}\\
&\leq
\frac{n^2}{4}\cdot \frac{C_3}{n} ~+~ 4n \log n \cdot C_1
~\leq~ Cn \log n,
\end{split}
\end{equation*}
for an absolute constant $C>0$. \qed

\subsection{Proof of Theorem \ref{main}, Part (a)}\label{subsecmain1pf}

In this section, we conclude the proof of Theorem \ref{main}, part (a). Lemma \ref{1stmlem} says that 
	\begin{equation}\label{one}
	|\, \mathbb{E} [\Phi_{t+1} | \mathcal{F}_t] - (1-\gamma) \Phi_t\,| \leq  \frac{3\pi}{4n},
	\end{equation}
	where $\gamma = \pi^2/n^2 - O(n^{-4})$.
Moreover,  Lemma \ref{2ndmlem} gives us that 
\begin{equation}\label{two}
	\mathbb{E} [(\Delta \Phi_t)^2 | \mathcal{F}_t] \leq Cn \log n.
	\end{equation}
	
\noindent Also, by the definition of $\Phi_t$, when $t=0$ it satisfies that
\begin{equation}\label{starting}
\Phi_0 = \sum_{x=1}^{n-1}\frac{1}{2} \{x \wedge (n-x)\} \sin \left(\frac{\pi x}{n}\right) \geq 2\sum_{x= \frac{n}{4}}^{\frac{n}{2}} \frac{n}{4} \sin \left( \frac{\pi}{4} \right) \geq \frac{n^2}{8 \sqrt{2}}.
\end{equation}

\noindent Define $\Phi : S_n \rightarrow \mathbb{R}$ to be
$$\Phi(\sigma) = \sum_{x=1}^n h(\sigma, x) \sin \left(\frac{\pi x}{n} \right), $$
where $h(\sigma, \cdot)$ is the height function of $\sigma$ defined in (\ref{heightdef}). Note that $\Phi(\sigma_{(n-1)t}) = \Phi_t$. Plugging $\Phi$ into Lemma \ref{W}, the equations \eqref{one}, \eqref{two} and \eqref{starting} imply that
\begin{equation*}
t_{mix} (1- \varepsilon) \geq \frac{n^3}{2 \pi^2} \log n - \frac{n^3}{2 \pi^2} \log \left(\frac{c \log n}{\varepsilon} \right),
\end{equation*}
where $c>0$ is a universal constant.
\qed

\section{Following one card}\label{seconec}
Throughout this section, we label our deck of $n$-cards  by $	[n_0]:= \{0,1, \ldots , n-1 \}$. Our state space is the symmetric group on $[n_0]$, which is denoted by $S_{n_0} $. For $a \in [n_0]$, let  $q_t(a)= \sigma_{(n-1)t}^{-1}(a) $ denote the position of card $a$ at time $(n-1)t$.

Let $T_a : = \min\{t: q_t(a) = n-1 \}$ be the first (scaled) time that the card $a$ in the deck reaches at the right end.  Our goal in this subsection is to prove the following lemma on $T_a$:

\begin{lemma}\label{couptimelem}
	Let $a$ be an arbitrary element of $[n_0]$, and define $q_t(a)$, $T_a$ as above. For any CAT shuffle $(\sigma_t)$ and any  $ \delta>0 $, there exist  $N_\delta, \: \theta_\delta>0$ independent of $n$ such that for all  $n\geq N_\delta$ and $ \theta_\delta \leq \theta \leq n$, we have
	\begin{equation}\label{couptimeineq}
	\mathbb{P} \left(T_a > \frac{\theta n^2}{\pi^2} \right) \leq  (1+O(\theta n^{-2}))e^{- \frac{(1-\delta)}{4}\theta}.	
	\end{equation}
\end{lemma}

In order to prove Lemma \ref{couptimelem}, we analyze the process $q_t(a)$ by coupling it with another random walk that we may have a better control. From now on, we focus on the process $\{ q_{t}(a)\}_{t\in \mathbb{N}}$, regarding each exploration of the whole line as a single step.  Let $X$ be a random variable on $\mathbb{Z}$ with the following probability distribution:
\begin{enumerate}
	\item[$\bullet$] For all $k\in \mathbb{Z}$, $\mathbb{P}(X=k) = 2^{-(|k|+3)} + 2^{-(3-|k|)}\textbf{1}_{\{|k|\leq 1 \}} $.
\end{enumerate}
Note that $X$ has mean $0$ and variance $2$. Let $X_i$ be i.i.d. copies of $X$, and define $S_t$ to be
\begin{equation*}
S_t := \sum_{i=1}^t X_i.
\end{equation*}

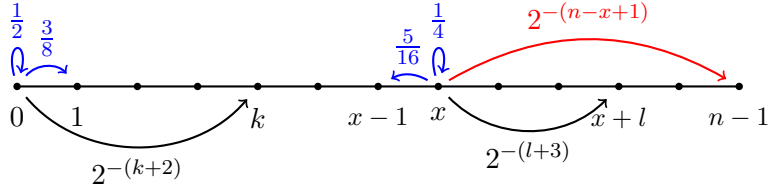
\begin{figure}\label{jumpfig}
	\centering
	\begin{tikzpicture}[thick,scale=1, every node/.style={transform shape}]
	\foreach \x in {0,...,12}{
		\filldraw[black] (0.8*\x,0) circle (1pt);
	}
	\draw (0,0)--(9.6,0);	
	\node[label=below:{$0$}] (n1) at (0,0) {};
	\node[label=below:{$1$}] (n2) at (0.8,0) {};
	\node[label=below:{$k$}] (n3) at (3.2,0) {};
	\node[label=below:{\small{$x-1$}}] (n4) at (4.8,0) {};
	\node[label=below:{$x$}] (n5) at (5.6,0) {};
	\node[label=below:{\small{$x+l$}}] (n6) at (8,0) {};
	\node[label=below:{\small{$n-1$}}] (n7) at (9.6,0) {};
	\draw[every loop]
	(n1) edge[blue, bend left=50, auto=left] node {$\frac{3}{8}$} (n2)
	(n1) edge[ bend right=50, auto=right] node {$2^{-(k+2)} $} (n3)
	(n1) edge[blue, loop above=10, auto=above] node {$\frac{1}{2}$} (n1)
	(n5) edge[blue, bend right, auto=right] node {$\frac{5}{16}$} (n4)
	(n5) edge[ bend right=46, auto=right] node {$2^{-(l+3)}$} (n6)
	(n5) edge[red, bend left, auto=left] node {$2^{-(n-x+1)}$} (n7)
	(n5) edge[blue, loop above, auto=above] node {$\frac{1}{4}$} (n5);
	\end{tikzpicture}
	\caption{Jump probabilities of the process $\{q_{t}(a)\}$. }
	\vspace{4mm}
\end{figure}

\begin{lemma}\label{rwcouplem}
	For all $a\in[n_0]$, there is a coupling between $\{ q_{t}(a)\}_{t\in \mathbb{N}}$ and $\{X_i\}_{i\in \mathbb{N}}$ such that for all $t\geq 0$,  on the event $\{ T_a > t \}$ we have
	\begin{equation*}
	q_{t}(a)  ~\geq~ \widehat{S}^a_t~:=~ S_t - (\min \{S_s: s \leq t \} \wedge (-q_0(a)) ).
	\end{equation*}
\end{lemma}

\begin{remark}\label{couplingrmk}
	\textnormal{$\widehat{S}^a_t$ is obtained by pushing $S_t+ q_0(a)$ above as little as possible while making it stay non-negative.}
\end{remark}

\begin{proof}[Proof of Lemma \ref{rwcouplem}]
	We first notice that the distribution of $q_{t}(a) - q_{t-1}(a)$ is very similar to that of $X$, as drawn in Figure 1. Given that $0< x:= q_{t-1}(a) <n-1 $, one can see that
	\begin{enumerate}
		\item [$\bullet$] For $-x+1 \leq k  \leq n-x-2 $, 
		\begin{equation}\label{jumpprob1}
		\mathbb{P} ( q_{t}(a) = x+k) = 2^{-(|k|+3)} + 2^{-(3-|k|)} \textbf{1}_{\{|k|\leq 1\}} = \mathbb{P}(X=k);
		\end{equation}
		
		\item [$\bullet$] For $k\in \{-x, \;n-x-1 \}$, $\mathbb{P} (q_{t}(a) = x+k) = 2^{-(|k|+2)} + 2^{-(3-|k|)} \textbf{1}_{\{|k|\leq 1\}}$.
	\end{enumerate}
	If $x = 0$, then
	\begin{enumerate}
		\item [$\bullet$] For $0\leq k \leq n-2$, 
		\begin{equation}\label{jumpprob2}
		\mathbb{P} (q_{t}(a) = k) =  2^{-(k+2)} + \frac{1}{4} \textbf{1}_{\{k\leq 1 \}} \geq \mathbb{P}(X=k);
		\end{equation}
		
		\item [$\bullet$] For $k=n-1$, $\mathbb{P} (q_{t} (a) = k) = 2^{-n}$.
	\end{enumerate}
	
	Notice that if $0<x:= q_{t-1}(a)<n-1$, we have $\mathbb{P}(q_{t}(a) = 0) = \mathbb{P}(X \leq -x)$. Combined with (\ref{jumpprob1}), this implies that when $0<x<n-1$,  the laws of $q_{t}(a)$ and $X_t$ can be coupled so that
	\begin{equation}\label{couplingeq1}
	q_{t}(a)-x = X_t \vee (-x).
	\end{equation} 
	Similarly when $x=0$, we have $\mathbb{P}(q_{t}(a)=0) \leq  \mathbb{P}(X \leq 0)$, and hence with (\ref{jumpprob2}) gives us that we have a coupling of $q_{t}(a)$ and $X_t$ that satisfies
	\begin{equation}\label{couplingeq2}
	q_{t}(a)-x \geq X_t \vee (-x).
	\end{equation}
	
	\noindent As mentioned in Remark \ref{couplingrmk}, $\widehat{S}^A_t$ is the process obtained by pushing $S_t + q_0(a)$ above to $0$ whenever it hits a negative point. Therefore, under the aforementioned coupling (\ref{couplingeq1}, \ref{couplingeq2}), $q_{t}(a)$ and $\widehat{S}^a_t$ satisfies 
	\begin{equation*}
	q_{t}(a) \geq \widehat{S}_t.
	\end{equation*}
\end{proof}

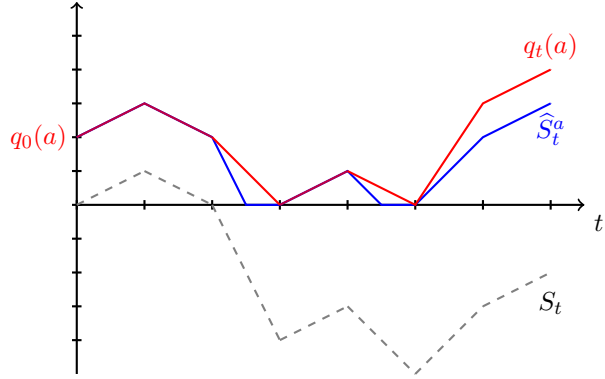
\begin{figure}
	\centering
	\begin{tikzpicture}[thick,scale=.9, every node/.style={transform shape}]
	\foreach \x in {0,...,7}{
		\draw (\x cm, 2pt) -- (\x cm, -2pt);
		}
	\foreach \x in {-4,...,5}{
		\draw (2pt, .5*\x cm)-- (-2pt, .5*\x cm);
		}
	\draw[thick,->] (0,0) -- (7.5,0) node[anchor=north west] {$t$};
	\draw[thick,->] (0,-2.5) -- (0,3) ;
	\draw[dashed,gray] (0,0) -- (1,.5);
	\draw[dashed,gray] (1,.5) -- (2,0);
	\draw[dashed,gray] (2,0) -- (3,-2);
	\draw[dashed,gray] (3,-2) -- (4,-1.5);
	\draw[dashed,gray] (4,-1.5) -- (5,-2.5);
	\draw[dashed,gray] (5,-2.5) -- (6,-1.5);
	\draw[dashed,gray] (6,-1.5) -- (7,-1) ;
	 \node[label=below:{$S_t$}] (n7) at (7,-1) {};
	
	\draw[thick,blue] (0,1) -- (1,1.5)--(2,1)--(2.5,0)--(3,0)--(4,0.5)--(4.5,0)--(5,0)--(6,1)--(7,1.5);
	 \filldraw[blue] (7,1.5) circle (0pt) node[blue, anchor=north] {$\widehat{S}_t^a $};
	  \filldraw[red] (7,2) circle (0pt) node[red, anchor=south] {$q_{t}(a) $};
	    \filldraw[red] (0,1) circle (0pt) node[red, anchor=east] {$q_{0}(a) $};
	 \draw[thick, purple] (0,1) -- (1,1.5)--(2,1);
	 \draw[thick, red] (2,1) -- (3,0);
	 \draw[thick,purple] (3,0)--(4,0.5);
	 \draw[thick, red] (4,0.5)--(5,0)--(6,1.5)--(7,2);
	\end{tikzpicture}
	\caption{Sample paths of $q_{t}(a)$ and $\widehat{S}^a_t := S_t - (\min_{s\leq t} S_s \wedge (-q_0(a)))$, with $q_0(a)=2$. }
\end{figure}
Let us define the stopping time $\tau_n^x$ such that
 $$\tau_n^x := \min\{t:  S_t - (\min_{s\leq t} S_s \wedge (-x)) \geq n\}. $$
Due to Lemma \ref{rwcouplem}, it suffices to prove the corresponding inequality for $\tau_n^x$ as (\ref{couptimeineq}) for arbitrary $x$. Since $S_t - (\min_{s\leq t} S_s \wedge (-x)) $ is increasing in $x$, it is enough to look at $\tau_n^0$.  Consider the extension $S_t$ to non-integer $t$'s by setting $(t, S_t)$ to be the point on the linear segment connecting $(r, S_r)$ and $(r+1, S_{r+1})$ for $r= \lfloor t \rfloor$. Since the increments of $S_t$ for integer times are i.i.d. with mean $0$ and variance $2$, Donsker's theorem (see e.g., \cite{Durrett}) directly implies the following:

\begin{proposition}\label{convprop}
 Let $B_t$ and $R_t$ denote the standard Brownian motion and the standard reflected Brownian motion, respectively. Then, as $m \rightarrow \infty$, we have the following weak convergence when viewed as measures on $C[0,\infty)$, the space of continuous functions on $[0,\infty)$:
	\begin{equation*}
	\begin{split}
	\frac{S_{m^2 t}}{\sqrt{2}m} ~~~~~~~~~~~&\longrightarrow ~~~~B_t;\\
	\frac{S_{m^2 t} -\min_{s\leq m^2t} S_s }{\sqrt{2}m} ~~&\longrightarrow ~~~~R_t.
	\end{split}
	\end{equation*}
\end{proposition}

\begin{proof}
	The first equation is a restatement of Donsker's theorem. For the second part, define $\Psi : C[0,\infty) \rightarrow C[0,\infty)$ as $\Psi(f)(t) = f(t)- \min_{s \leq t} \{f(s)\}  $. Since $\Psi(B_t) = R_t$ in law and $\Psi$ is continuous with respect to the supremum norm topology, the second convergence follows from the first one. 
\end{proof}

Define $\tau^R := \min\{t>0: R_t \geq 1/\sqrt{2}\}$. As an immediate consequence of Proposition \ref{convprop}, we have the following corollary.

\begin{corollary}\label{convcor}
	For any constant $\theta>0$ we have
	\begin{equation}\label{taulim}
	\lim_{n\rightarrow \infty} \mathbb{P}(\tau_n^0 > \theta n^2 ) = \mathbb{P}(\tau^R > \theta )
	= 
	\mathbb{P} (\tau^{|B|} > \theta),
	\end{equation}
	where $\tau^{|B|} := \min \{t>0 : |B_t| \geq 1/\sqrt{2} \}$.
\end{corollary}

\begin{proof}
	The first equality is obvious by Proposition \ref{convprop}. The second follows by the fact that $(R_t)_{t\geq 0} = (|B_t|)_{t \geq 0}$ in law (see e.g., Chapter 3.6 of \cite{KS}).
\end{proof}
From now on, we choose to look at $\tau^{|B|}$ instead of $\tau^R$. Let $\widetilde{S}_m$ denote the simple random walk so that the increments are i.i.d. $2$Ber$(\frac{1}{2})-1$ and $\widetilde{S}_0 = 0$. Let $\widetilde{\tau}^n := \min\{ m: \widetilde{S}_m \notin (-n/\sqrt{2} , n/\sqrt{2})  \}$. Then the following lemma is based on the same spirit as   Lemma 9 of \cite{Wilson}. We postpone its proof to \S\ref{subsecwilson}.

\begin{lemma}\label{wilsoncouptimelem}
	There exists a constant $C>0$ that satisfies 
	\begin{equation*}
		\mathbb{P} (\widetilde{\tau}^n > \theta n^2) < C(1+ O(\theta n^{-2})) \exp \left(-\frac{\pi^2 \theta}{4} \right)
	\end{equation*}
 for all $\theta >0$ ($\theta$ may depend on $n$).
\end{lemma} 

The following Corollary is a consequence of Corollary \ref{convcor} and Lemma \ref{wilsoncouptimelem}.

\begin{corollary}\label{couptimecor}
	For any $ \delta >0$, there exist constants $\theta_0 = \theta_0(\delta)$ and $N=N(\delta, \theta_0)>0$  such that for all $n\geq N$, we have
	 \begin{equation*}
	 \mathbb{P} \left(\tau^0_n > \theta_0 n^2 \right) \leq (1+O(\theta_0 n^{-2})) \exp \left(-\frac{(1-\delta)}{4} \pi^2 \theta_0  \right).
	 \end{equation*}
\end{corollary}

\begin{proof}
	Let $\delta>0$ be given. By Lemma \ref{wilsoncouptimelem}, we can pick a large $\theta_0 = \theta_0(\delta)$ such that for all constants $\theta \geq \theta_0 -\delta$ not depending on $n$,
	\begin{equation*}
	\mathbb{P}(\widetilde{\tau}^n > \theta n^2) \leq (1+ O(\theta n^{-2})) \exp \left(-\frac{(1-\delta/2)}{4} \pi^2 \theta\right).
	\end{equation*}
	Then Donsker's theorem implies that $\mathbb{P} ( \tau^{|B|} > \theta_0 ) \leq \exp(-(1-\delta/2)\pi^2 \theta_0)$, since $\widetilde{S}_{m^2t} / m$ converges to $(B_t)$ as in Proposition \ref{convprop}. Noting that $\tau^R$ and $\tau^{|B|}$ share the same law, we use (\ref{taulim}) to deduce that there exists $N=N(\delta, \theta_0)$ such that for all $n>N$,
	\begin{equation*}
	\mathbb{P}(\tau_n^0 > \theta_0 n^2) \leq \mathbb{P} ( \tau^R > \theta_0 - \delta) \leq (1+O(\theta_0 n^{-2})) \exp \left(-\frac{(1-\delta)}{4}\pi^2 \theta_0\right),
	\end{equation*}
	which is the desired inequality.
\end{proof}

\noindent\textit{Proof of Lemma \ref{couptimelem}}.  
As observed in Lemma \ref{rwcouplem}, we can couple the two processes $q_t(a)$ and $S_t$ such that $T_a \leq \tau_n^0 $, as a single increment in $S_t$ corresponds to $n-1$ steps of swapping in the CAT shuffle. Therefore, Corollary \ref{couptimecor} implies that
\begin{equation}\label{couptimeineq1}
\mathbb{P}(T_a > \theta_0 n^2) \leq (1+O(\theta_0 n^{-2}))\exp\left(-\frac{(1-\delta/2)}{4}\pi^2 \theta_0 \right)
\end{equation}
for some constant $\theta_0>0$ depending on $\delta$. For any $\theta>\theta_0$ and $n>N_\delta$, we have 
\begin{equation*}
\begin{split}
\mathbb{P}(T_a > \theta n^2) \leq& (1+O(\theta n^{-2})))\exp\left(-\frac{(1-\delta/2)}{4}\pi^2 \theta_0 \left\lfloor \frac{\theta}{\theta_0} \right\rfloor \right) \\
\leq& (1+O(\theta n^{-2})) \exp\left(-\frac{(1-\delta)}{4}\pi^2 \theta\right),
\end{split}
\end{equation*}
where the first inequality is obtained by iterating (\ref{couptimeineq1}) for $\lfloor \theta/\theta_0  \rfloor$-times and by the fact that $\theta \leq n$, and the last inequality holds for all $\theta_\delta <\theta \leq n$ where $\theta_\delta$ is a large constant. \qed

\section{The upper bound}\label{secupbd}

In \cite{Lacoin}, Lacoin derives the sharp upper bound on the mixing time of the random AT shuffle by introducing a monotone coupling of the model and implementing the censoring inequality in a clever way. They define the function $\widetilde{\sigma}_t$ which can be understood as the ``height'' of $\sigma_t$ to build up a monotone framework of the system. It turns out that a similar argument is also applicable to the CAT shuffle along with some appropriate adjustments.

However in the CAT shuffle, the major difficulty of adopting this argument comes from understanding the decay of the height $\widetilde{\sigma}_t$. In \cite{Lacoin}, exponential decay of $\widetilde{\sigma}_t$ is obtained using algebraic properties of the model based on its eigenvalues and eigenfunctions. 
Since this approach seems impossible  in the current context, we rely on the ideas developed in \S \ref{seconec} to deduce the same property for the CAT shuffle.

In \S\ref{subsecmono}, we introduce a monotone coupling for the CAT shuffle. In such a monotone system, we can take advantage of the censoring inequality, which essentially says that if we ignore some updates (swaps) in a CAT shuffle, then the distance from equilibrium of the resulting chain is greater than that of the original one.  In \S\ref{subsecpfmain2} we conclude the proof of Theorem \ref{main} based on the tools from the previous sections. Finally, we prove the decay estimate on $\widetilde{\sigma}_t$ in \S \ref{subsectildebd}.

\subsection{Monotone coupling}\label{subsecmono}

In this subsection, we introduce a monotone coupling for the CAT shuffle following the argument of Lacoin \cite{Lacoin} and Wilson \cite{Wilson}. Via the monotone coupling, we derive the censoring inequality which will be crucial in \S\ref{subsecpfmain2}. 

Let $S_{n}$ be the group of permutations on $[n] = \{ 1,\ldots,n\}$. For each $\sigma \in S_n$, we define the function $\widetilde{\sigma} : [n] \times [n] \rightarrow \mathbb{R}$ as follows:
\begin{equation}\label{sigtildef}
\widetilde{\sigma}(x,y) := \sum_{z=1}^x \textbf{1}_{\{\sigma(z) \leq y \}} - \frac{xy}{n}. 
\end{equation}
Subtracting $xy/n$ is introduced in order to set the average  $\mu(\widetilde{\sigma}(x,y))$ to be $0$, where $\mu$ denotes the uniform measure on $S_n$. Throughout this section, we use the following partial order on $S_n$ based on the function $\widetilde{\sigma}$:
\begin{center}
	For $\sigma,\sigma' \in S_n$, ~$\sigma \geq \sigma'$ ~~~if and only if ~~~ $\widetilde{\sigma}(x,y) \geq \widetilde{\sigma}' (x,y)$ ~ for all $x,y \in [n]$.
\end{center}
Under this ordering, one can observe that the identity element (denoted id) is maximal, and the permutation that maps $x$ to $n+1-x$ is minimal. 

\begin{definition}[Monotone coupling]\label{defmonocoup}
	\textnormal{Let $ \{U_t : t\in \mathbb{N}\}$ be the family of i.i.d.\,Ber$(\frac{1}{2})$ random variables, and let $\{\textbf{c}_i : i \in \mathbb{N} \}$ be the family of i.i.d.\,Unif$\{1,n\}$ random variables (i.e., $\textbf{c}_i = 1$ or $n$, each with probability half) that is independent from $U_t$'s.  Using $\textbf{c}_i$'s and $U_t$'s, we define the updates of the CAT shuffle as follows:
	\begin{enumerate}
		\item [$(1)$] At time $(n-1)i +1$ for each $i=0,1,... ,$ we begin the exploration starting from position $\textbf{c}_i$. That is, if for instance $\textbf{c}_i=1$, then during the time interval from $(n-1)i+1$ to $(n-1)(i+1)$, we explore the deck from left to right. 
		\item [$(2)$] Suppose that $(x, x+1) $ is the edge we are about to swap or not at time $t+1$. 
		\begin{enumerate}[leftmargin=.3cm]
			\item[$\bullet$] If either $U_t = 0$ and $\sigma_t(x) < \sigma_t(x+1)$ or $U_t =1$ and $\sigma_t(x) > \sigma_t(x+1)$, then we swap the edge $(x,x+1)$, hence obtaining $\sigma_{t+1}(x)=\sigma_t(x+1)$ and $\sigma_{t+1}(x+1)=\sigma_t(x)$.
			\item [$\bullet$] In other cases, we do nothing.
		\end{enumerate}
	\end{enumerate}
}
\end{definition}

In other words, if $U_t=0$ we reverse-sort the cards at positions $x, x+1$, whereas  if $U_t =1$ we sort the cards at $x,x+1$. One can easily check that this update rule exhibits the same transition matrix as the CAT shuffle. The following proposition describes a significant advantage of this coupling, namely  the preservation of monotonicity. For a proof, we refer to \cite{Lacoin}.

\begin{proposition}[\cite{Lacoin}, Proposition 3.1]\label{monoprop}
	Let $\xi, \xi' \in S_n$ and let $\sigma_t^\xi$ (resp. $\sigma_t^{\xi'}$) denote the CAT shuffle starting from $\xi$ (resp. $\xi'$) coupled by the aforementioned update rules. If $\xi \geq \xi'$, then we have
	\begin{equation*}
	\sigma_t^\xi \geq \sigma_t^{\xi'}~~~~ \textnormal{for all } t\geq 0.
	\end{equation*}
\end{proposition}

\begin{definition}
A probability distribution $\nu$ on $S_n$ is called \textit{increasing} if $\nu(\sigma) \geq \nu(\sigma')$ holds for all $\sigma, \sigma' \in S_n$ such that $\sigma \geq \sigma'$. 
\end{definition}
One property of the adjacent transposition shuffle is that it preserves the monotonicity of measures. This fact is formalized in the following lemma, whose proof can be found in \cite{Lacoin}.

\begin{lemma}[\cite{Lacoin}, Proposition A.1]\label{increasinglem}
	Let $\nu$ be an increasing probability measure on $S_n$. For any $x\in [n-1]$, let $\sigma^x$ be the resulting state of $\sigma$ after performing an update at edge $(x,x+1)$, i.e., either swap the labels $\sigma(x), \sigma(x+1)$ with probability half or stay fixed otherwise. Let $\nu^x$ denote the distribution of $\sigma^x$ when $\sigma \sim \nu$. Then, $\nu^x$ is increasing.
\end{lemma} 

Furthermore, we introduce two additional tools which will be used in the next subsection: decay estimate of $\widetilde{\sigma}_t$ and the censoring inequality.

\vspace{3mm}
For any fixed $x,y \in [n]$, the average of $\widetilde{\sigma}(x,y)$ over $\mu$ is $0$. We are interested in decay speed of the expected value of $\widetilde{\sigma}_t (x,y)$, which can be described as the following lemma:

\begin{lemma}\label{tildebdlem}
	Let $(\sigma_t)$ denote the CAT shuffle on $[n]$ that starts from an arbitrary initial state  and  let $\delta>0$ be arbitrary. Then there exist $N_\delta, \:\theta_\delta>0$ independent of $n$ such that  for any $n \geq N_\delta $, $x,y \in [n]$, and  $ \theta_\delta n^3 <t \leq n^4$,  we have
	\begin{equation*}
	|\, \mathbb{E} [\widetilde{\sigma}_t (x,y) ]\,| \leq n(1+ O(tn^{-5})) \exp \left(-(1-\delta)\frac{\pi^2}{n^3} t \right).
	\end{equation*} 
\end{lemma}

\begin{remark}
	\textnormal{The random AT shuffle version of Lemma \ref{tildebdlem} is discussed in \cite{Lacoin}, Lemma 4.1. In the random AT shuffle, this is proven by a direct computation of the eigenvalues and eigenvectors of the simple random walk. In the present context, such method seems extremely difficult to be applied because the model is more complicated. Instead, we choose an alternative approach, based on the ideas similar to  Lemmas  \ref{couptimelem} and \ref{wilsoncouptimelem}. Due to its technicality, we defer the proof of Lemma \ref{tildebdlem} to \S\ref{subsectildebd}.}
\end{remark}

\vspace{3mm}

In \cite{PW13}, Peres and Winkler proved the censoring inequality for the  Glauber dynamics on monotone spin systems. The message of this inequality is that ignoring updates can only slow down the mixing. In \cite{Lacoin}, Lacoin extended the inequality to the random AT shuffle. It turns out that in the CAT shuffle, the censoring inequality is still true.

To formalize, a \textit{censoring scheme} is a function $\mathcal{C} : \mathbb{N} \rightarrow \mathcal{P}([n-1])$ that is interpreted as follows:
\begin{center}
	$\bullet$ At each time $t$, the edge $(x,x+1)$ that we are about to update is ignored if and only if $x \in \mathcal{C}(t)$.
\end{center}
Let $P_\nu^t$ be the probability distribution of the CAT shuffle at time $t$ with initial distribution $\nu$, and let  $P_{\nu, \mathcal{C}}^t$ denote the distribution of the CAT shuffle at time $t$  which has performed  the censoring dynamics according to $\mathcal{C}$ while started from the same distribution $\nu$.  Intuitively, the updates are the triggers that carry the chain to its equilibrium, so one might guess that the censored dynamics is further away from the equilibrium than the original one. This intuition turns
out to be true for the CAT shuffle due to monotonicity of the system, as long as we start from an ``increasing'' initial distribution $\nu$. \cite{Lacoin} and \cite{PW13} describe this phenomenon as follows.

\begin{proposition}[The censoring inequality]
	Let $\nu$ be an increasing probability distribution on $S_n$. For any censoring scheme $\mathcal{C}: \mathbb{N} \rightarrow \mathcal{P}([n-1])$ and any $t\geq 0$, we have
	\begin{equation}\label{cenineq}
	|| P_{\nu}^t - \mu ||_{T.V.} \leq || P_{\nu, \mathcal{C}}^t -\mu ||_{T.V.}.
	\end{equation}
	In particular, (\ref{cenineq}) holds for the starting distribution $\nu = \delta_{\id}$, the point mass at the identity.
\end{proposition} 

\noindent We omit the proof of the censoring inequality. A proof can be found either in \cite{Lacoin} or in \cite{PW13}.

\subsection{Proof of Theorem \ref{main}, Part (b)}\label{subsecpfmain2}
In this subsection we prove the second part of Theorem $\ref{main}$. Implementing the ingredients we obtained in the previous sections, the proof follows similarly as in the case of random AT shuffle \cite{Lacoin}. 

To this end, we first explain the projection of measures on $S_n$ which serves as a pretty tool to understand the mixing clearly. After that we describe the main ideas of the proof. Some of the details will be presented at Appendix.

Let $K$ be a fixed integer and define $x_i := \lfloor \frac{in}{K} \rfloor$ for all $i=0,1,\ldots,n$. Following the notations in \cite{Lacoin}, we define the functions $\widehat{\sigma}$ and $\bar{\sigma}$ for each $\sigma \in S_n$ and the sets $\widehat{S}_n$, $\bar{S}_n$ by
\begin{equation*}
\begin{split}
&\widehat{\sigma} : [n] \times [K] \rightarrow \mathbb{R}, \;~~~\widehat{\sigma}(x,j) := \widetilde{\sigma}(x, x_j), ~~~~\widehat{S}_n := \{ \widehat{\sigma}: \sigma \in S_n \};\\
&\bar{\sigma} : [K] \times [K] \rightarrow \mathbb{R}, ~~~\bar{\sigma}(i,j) := \widetilde{\sigma}(x_i, x_j), ~~~~\bar{S}_n := \{ \bar{\sigma}: \sigma \in S_n \}.
\end{split}
\end{equation*}

That is, we are intentionally forgetting information from $\widetilde{\sigma}$ by projecting it to a smaller domain.  For a probability measure $\nu$ on $S_n$, we similarly define the measure $\widehat{\nu}$ (resp. $\bar{\nu}$) on $\widehat{S}_n$ (resp. $\bar{S}_n$) by 
\begin{equation*}
\widehat{\nu}(\widehat{\sigma}) := \sum_{\xi: \;\widehat{\xi} = \widehat{\sigma}}\nu(\xi);
~~~~
\bar{\nu}(\bar{\sigma}) := \sum_{\xi: \;\bar{\xi} = \bar{\sigma}}\nu(\xi).
\end{equation*}

Furthermore, we introduce one more notation which is closely related to the projection $\widehat{\nu}$. Let $T_n$ be the subset of $S_n$ defined as
\begin{equation*}
T_n := \{\sigma \in S_n : \; \sigma(\{x_{i-1}+1, \ldots, x_i \}) = \{x_{i-1}+1 , \ldots , x_i \} ~~ \textnormal{for all } i\in [K] \}.
\end{equation*} 

\noindent It is clear that $|T_n|  = \prod_{i=1}^{K} (\Delta x_i)!$, where $\Delta x_i := x_i - x_{i-1}$. For a probability measure $\nu$ on $S_n$, the probability measure $\nu^{u}$ is defined by
\begin{equation*}
\nu^u ( \sigma) := \frac{1}{|T_n|} \sum_{\tau \in T_n} \nu (\tau \circ \sigma).
\end{equation*}

\noindent Therefore, $\nu^u$ becomes an invariant measure under composing an element of $T_n$. In other words,
it is locally uniformized in the sense that permuting the label $\sigma(x) \in (x_{i-1}, x_i]$ within the same interval $(x_{i-1},x_i]$ does not affect its probability.
In addition, note that for any $\sigma\in S_n$ and $\tau \in T_n$, $\widehat{\sigma} = \widehat{\tau \circ \sigma}$. Based on this observation, we can deduce a connection between $\widehat{\nu}$ and $\nu^u$.

\begin{lemma}[\cite{Lacoin}, Lemma 4.3]\label{hatulem}
	Let $\mu$ denote the uniform measure on $S_n$. For any probability measure $\nu$ on $S_n$, 
	\begin{equation*}
	||\widehat{\nu} - \widehat{\mu} ||_{T.V.} =
	||\nu^u - \mu||_{T.V.}.
	\end{equation*}
\end{lemma}

\begin{proof}
	The lemma readily follows from the above observation. Since $\nu^u$ is constant on $\{\sigma: \widehat{\sigma} = \widehat{\xi} \}$ for each fixed $\widehat{\xi} \in \widehat{S}_n$, we have
	\begin{equation*}
	\begin{split}
	\sum_{\sigma} |\nu^u (\sigma) - \mu(\sigma)|
	&= \sum_{\widehat{\xi} \in \widehat{S}_n} \;\,
	 \left| \sum_{\sigma:\;\widehat{\sigma} = \widehat{\xi}}
	\Big( {\nu^u}({\sigma}) - \mu(\sigma) \Big) \right| \\
	&=
	\sum_{\widehat{\xi} \in \widehat{S}_n} \;\left|
	  \sum_{\sigma:\;\widehat{\sigma} = \widehat{\xi}}
	\Big( {\nu}({\sigma}) - \mu(\sigma) \Big)  \right|
	=
	\sum_{\widehat{\xi} \in \widehat{S}_n} | \widehat{\nu}(\widehat{\xi}) - \widehat{\mu}(\widehat{\xi}) |.
	\end{split}
	\end{equation*}
\end{proof}

In order to establish the main theorem, we will introduce a censoring scheme $\mathcal{C}$, and show that the censored dynamics indeed mixes in the desired time, and hence impying the mixing of the original chain by the censoring inequality. We follow \cite{Lacoin} for the construction of $\mathcal{C}$, while the proofs for each step rely on different ingredients to fit with the CAT shuffle. 

Let $\eta>0$ be a small fixed constant, set $K : = \lfloor \eta^{-1} \rfloor$ and let $x_i := \lfloor in/K\rfloor$ as before. Define the censoring scheme $\mathcal{C}: \mathbb{N} \rightarrow \mathcal{P}([n-1])$ by
\begin{equation*}
\mathcal{C}(t) =
\begin{cases}
\{x_i :\; i\in [K-1] \} , & \textnormal{if }\; t\in[0,t_1] \cup [t_2,t_3]; \\
~~\emptyset & \textnormal{if }\; t \in (t_1,t_2),
\end{cases}
\end{equation*}

\noindent where the times $t_1, t_2$ and $t_3$ are given by
\begin{equation*}
t_1:= \left(\frac{\eta}{3}\right) \frac{n^3}{2\pi^2} \log n;
~~~~
t_2:=\left( 1+\frac{2\eta}{3} \right) \frac{n^3}{2\pi^2} \log n;
~~~~
t_3:=\left( 1+\eta \right) \frac{n^3}{2\pi^2} \log n.
\end{equation*}
In other words, in the first and the third steps, we ignore the updates happening at edges $(x_i, x_i+1)$ for all $i\in[K-1]$, while running the chain without censoring in the second phase. Thus, in the first and the third steps, the chain operates separately at each interval $(x_{i-1}, x_i]$, while being dependent from each other since they share the directions of exploration $\{\textbf{c}_l \}$. What happens in the censored shuffle can intuitively be described as follows (also see Figure 3):

\begin{enumerate}
	\item [$(1)$] At time $t_1$, the cards in the same interval $(x_{i-1},x_i]$ are distributed nearly uniformly, hence becoming indistinguishable. Thus, we can label all cards in $(x_{i-1},x_i]$ by the \textit{same index} $i$. 
	\item [$(2)$] At time $t_2$, cards with different indices get mixed, and each interval $(x_{i-1}, x_i]$ contains approximately equal number of cards of index $j$ for all $j$. However, the locations within $(x_{i-1},x_i]$ of the cards of different indices might not be uniform.
	\item [$(3)$]  After time $t_3$, within each interval $(x_{i-1},x_i]$, the placement of cards of different indices become almost uniform.
\end{enumerate}

Let us denote the uniform measure by $\mu$ as before, and define
\begin{equation*}
\nu_t : = P^t_{\id, \mathcal{C}},
\end{equation*} 
the probability distribution of the censored CAT shuffle at time $t$ under the censoring scheme $\mathcal{C}$ which started from the initial state $\id$. Then by Lemma \ref{hatulem}, 
\begin{equation}\label{tvuhatineq}
|| \nu_t - \mu ||_{T.V.} \leq || \nu_t - \nu_t^u ||_{T.V.} + || \widehat{\nu}_t - \widehat{\mu}||_{T.V.}.
\end{equation} 

\noindent Having (\ref{tvuhatineq}) in mind, we will establish mixing in terms of $||\nu_t -\nu_t^u||$ and $||\widehat{\nu}_t - \widehat{\mu} ||$ as follows.

\begin{figure}
	\centering
	\begin{tikzpicture}[thick,scale=.7, every node/.style={transform shape}]
	\foreach \x in {0,...,17}{
		\filldraw[black] (\x,0) circle (1pt);
		\filldraw[black] (\x,-2.5) circle (1pt);
		\filldraw[black] (\x,-5) circle (1pt);
		\filldraw[black] (\x,-7.5) circle (1pt);
	}
	
	\filldraw[red] (0,0+0.1) circle (0pt) node[red, anchor=south] {\Large{$a_{1} $}};
	\filldraw[red] (1,0+0.1) circle (0pt) node[red, anchor=south] {\Large{$a_{2} $}};
	\filldraw[red] (2,0+0.1) circle (0pt) node[red, anchor=south] {\Large{$a_{3} $}};
	\filldraw[red] (3,0+0.1) circle (0pt) node[red, anchor=south] {\Large{$a_{4} $}};
	\filldraw[red] (4,0+0.1) circle (0pt) node[red, anchor=south] {\Large{$a_{5} $}};
	\filldraw[red] (5,0+0.1) circle (0pt) node[red, anchor=south] {\Large{$a_{6} $}};
	\filldraw[blue] (6,0+0.1) circle (0pt) node[blue, anchor=south] {\Large{$b_{1} $}};
	\filldraw[blue] (7,0+0.1) circle (0pt) node[blue, anchor=south] {\Large{$b_{2} $}};
	\filldraw[blue] (8,0+0.1) circle (0pt) node[blue, anchor=south] {\Large{$b_{3} $}};
	\filldraw[blue] (9,0+0.1) circle (0pt) node[blue, anchor=south] {\Large{$b_{4} $}};
	\filldraw[blue] (10,0+0.1) circle (0pt) node[blue, anchor=south] {\Large{$b_{5} $}};
	\filldraw[blue] (11,0+0.1) circle (0pt) node[blue, anchor=south] {\Large{$b_{6} $}};
	
	\filldraw[green] (12,0+0.1) circle (0pt) node[green, anchor=south] {\Large{$c_{1} $}};
	\filldraw[green] (13,0+0.1) circle (0pt) node[green, anchor=south] {\Large{$c_{2} $}};
	\filldraw[green] (14,0+0.1) circle (0pt) node[green, anchor=south] {\Large{$c_{3} $}};
	\filldraw[green] (15,0+0.1) circle (0pt) node[green, anchor=south] {\Large{$c_{4} $}};
	\filldraw[green] (16,0+0.1) circle (0pt) node[green, anchor=south] {\Large{$c_{5} $}};
	\filldraw[green] (17,0+0.1) circle (0pt) node[green, anchor=south] {\Large{$c_{6} $}};
	\foreach \x in {0,...,5}{
		\filldraw[red] (\x,-2.5+0.1) circle (0pt) node[red, anchor=south]{\Large{$a$}};
		\filldraw[blue] (\x+6,-2.5+0.1) circle (0pt) node[blue, anchor=south]{\Large{$b$}};
		\filldraw[green] (\x+12,-2.5+0.1) circle (0pt) node[green, anchor=south]{\Large{$c$}};
	}
	\foreach \x in {0,1}{
		\filldraw[red] (\x,-5+0.1) circle (0pt) node[red, anchor=south]{\Large{$a$}};
		\filldraw[red] (\x+6,-5+0.1) circle (0pt) node[red, anchor=south]{\Large{$a$}};
		\filldraw[red] (\x+12,-5+0.1) circle (0pt) node[red, anchor=south]{\Large{$a$}};
		\filldraw[blue] (\x+2,-5+0.1) circle (0pt) node[blue, anchor=south]{\Large{$b$}};
		\filldraw[blue] (\x+8,-5+0.1) circle (0pt) node[blue, anchor=south]{\Large{$b$}};
		\filldraw[blue] (\x+14,-5+0.1) circle (0pt) node[blue, anchor=south]{\Large{$b$}};
		\filldraw[green] (\x+4,-5+0.1) circle (0pt) node[green, anchor=south]{\Large{$c$}};
		\filldraw[green] (\x+10,-5+0.1) circle (0pt) node[green, anchor=south]{\Large{$c$}};
		\filldraw[green] (\x+16,-5+0.1) circle (0pt) node[green, anchor=south]{\Large{$c$}};
	}	
	
	\filldraw[red] (0,-7.5+0.1) circle (0pt) node[red, anchor=south]{\Large{$a$}};
	\filldraw[red] (4,-7.5+0.1) circle (0pt) node[red, anchor=south]{\Large{$a$}};
	\filldraw[red] (8,-7.5+0.1) circle (0pt) node[red, anchor=south]{\Large{$a$}};
	\filldraw[red] (10,-7.5+0.1) circle (0pt) node[red, anchor=south]{\Large{$a$}};
	\filldraw[red] (14,-7.5+0.1) circle (0pt) node[red, anchor=south]{\Large{$a$}};
	\filldraw[red] (16,-7.5+0.1) circle (0pt) node[red, anchor=south]{\Large{$a$}};
	
	\filldraw[blue] (1,-7.5+0.1) circle (0pt) node[blue, anchor=south]{\Large{$b$}};
	\filldraw[blue] (2,-7.5+0.1) circle (0pt) node[blue, anchor=south]{\Large{$b$}};
	\filldraw[blue] (7,-7.5+0.1) circle (0pt) node[blue, anchor=south]{\Large{$b$}};
	\filldraw[blue] (11,-7.5+0.1) circle (0pt) node[blue, anchor=south]{\Large{$b$}};
	\filldraw[blue] (12,-7.5+0.1) circle (0pt) node[blue, anchor=south]{\Large{$b$}};
	\filldraw[blue] (17,-7.5+0.1) circle (0pt) node[blue, anchor=south]{\Large{$b$}};	
	
	\filldraw[green] (3,-7.5+0.1) circle (0pt) node[green, anchor=south]{\Large{$c$}};
	\filldraw[green] (5,-7.5+0.1) circle (0pt) node[green, anchor=south]{\Large{$c$}};
	\filldraw[green] (6,-7.5+0.1) circle (0pt) node[green, anchor=south]{\Large{$c$}};
	\filldraw[green] (9,-7.5+0.1) circle (0pt) node[green, anchor=south]{\Large{$c$}};
	\filldraw[green] (13,-7.5+0.1) circle (0pt) node[green, anchor=south]{\Large{$c$}};
	\filldraw[green] (15,-7.5+0.1) circle (0pt) node[green, anchor=south]{\Large{$c$}};
	
	\draw[thick, ->] (8.5,-0.5) -- (8.5,-1.5);
	\draw[thick, ->] (8.5,-3) -- (8.5,-4);
	\draw[thick, ->] (8.5,-5.5) -- (8.5,-6.5);
	\draw[thick] (0,0) -- (5,0);
	\draw[thick] (6,0) -- (11,0);
	\draw[thick] (12,0) -- (17,0);
	
	\draw[thick] (0,-2.5) -- (5,-2.5);
	\draw[thick] (12,-2.5) -- (17,-2.5);
	\draw[thick] (6,-2.5) -- (11,-2.5);
	\draw[thick] (0,-5) -- (17,-5);
	\draw[thick] (0,-7.5) -- (5,-7.5);
	\draw[thick] (6,-7.5) -- (11,-7.5);
	\draw[thick] (12,-7.5) -- (17,-7.5);
	
	\draw[dashed] (5,0) -- (6,0);
	\draw[dashed] (11,0) -- (12,0);
	\draw[dashed] (5,-2.5) -- (6,-2.5);
	\draw[dashed] (11,-2.5) -- (12,-2.5);
	\draw[dashed] (5,-7.5) -- (6,-7.5);
	\draw[dashed] (11,-7.5) -- (12,-7.5);
	
	\foreach \x in {0,1,2,3}{
		\filldraw[black] (0,-2.5*\x-0.2) circle (0pt) node[anchor=north] {\LARGE$1$};
		\filldraw[black] (5,-2.5*\x-0.2) circle (0pt) node[anchor=north] {\LARGE$x_1$};
		\filldraw[black] (11,-2.5*\x-0.2) circle (0pt) node[anchor=north] {\LARGE$x_2$};
		\filldraw[black] (17,-2.5*\x-0.2) circle (0pt) node[anchor=north] {\LARGE$x_3$};
	}

	\filldraw[black] (17.5,0) circle (0pt) node[anchor=west] {\LARGE$t=0$};
	\filldraw[black] (17.5,-2.5) circle (0pt) node[anchor=west] {\LARGE$t=t_1$};
	\filldraw[black] (17.5,-5) circle (0pt) node[anchor=west] {\LARGE$t=t_2$};
	\filldraw[black] (17.5,-7.5) circle (0pt) node[anchor=west] {\LARGE$t=t_3$};
	\end{tikzpicture}
	\caption{An illustration of mixing divided into three steps. Dashed edges indicate the censored updates during the first and the third phases.}
\end{figure}
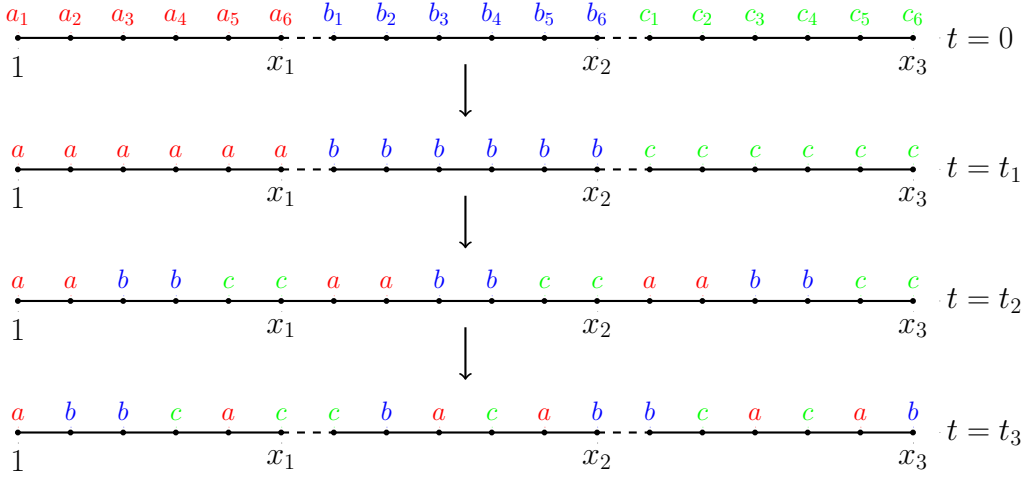

\begin{proposition}\label{t1prop}
	For any given $\eta, \varepsilon >0$, the following holds for all large enough $n$ and all $t>t_1$:
	\begin{equation*}
	|| \nu_t - \nu_t^u ||_{T.V. }\leq \varepsilon/3.
	\end{equation*} 
\end{proposition}

\begin{proposition}\label{t3prop}
	For any given $\eta, \varepsilon >0$, the following holds for all large enough $n$:
	\begin{equation*}
	|| \widehat{\nu}_{t_3} - \widehat{\mu} ||_{T.V. }\leq 2\varepsilon/3.
	\end{equation*} 
\end{proposition}

\vspace{4mm}

\noindent $\blacktriangleright$ \textit{Proof of Theorem \ref{main}, part (b) from Propositions \ref{t1prop}--\ref{t3prop}.} The censoring inequality and the equation (\ref{tvuhatineq}) implies that 
\begin{equation*}
||P^{t_3}_{\id} -\mu||_{T.V.} ~\leq~ ||\nu_{t_3} -\mu||_{T.V.} 
~\leq~
||\nu_{t_3} -\nu_{t_3}^u ||_{T.V.} + || \widehat{\nu}_{t_3} - \widehat{\mu} ||_{T.V.}
~\leq~
\varepsilon. 
\end{equation*}
Therefore the mixing time $t_{mix}(\varepsilon)$ of the CAT shuffle satisfies
\begin{equation*}
t_{mix}(\varepsilon) \leq t_3 = (1+\eta) \frac{n^3}{2\pi^2} \log n,
\end{equation*}
where $\eta>0$ can be taken arbitrarily small as $n$ tends to infinity. \qed

\vspace{5mm}
\noindent $\blacktriangleright$ \textit{Proof of Proposition \ref{t1prop}.} ~Our approach will be essentially the same as Lemma \ref{couptimelem}.  Let $\sigma_t$ be the state at time $t$ under performing the censoring of $\mathcal{C}$ with initial condition $\sigma_0=\id$. Since the cards can only move within each intervals $\{x_{i-1}+1,\ldots,x_i\}$ for $i\in[K-1]$ until time $t_1$, we have $\sigma_t \in T_n$ for all $t\leq t_1$. This implies that
\begin{equation*}
\nu_t^u = \textbf{1}_{\id}^u,
\end{equation*} 
where $\textbf{1}_{\id}$ is the point mass at $\id$. Moreover, $\textbf{1}_{\id}^u$ is the stationary distribution of our chain until time $t_1$. Therefore, by the coupling inequality,
\begin{equation}\label{cencoupeq}
||\nu_{t_1} - \nu_{t_1}^u||_{T.V.} = ||\nu_{t_1} - \textbf{1}_{\id}^u ||_{T.V.} \leq \max_{\tau, \tau' \in T_n} \mathbb{P}\left( \sigma_{t_1}^\tau \neq \sigma_{t_1}^{\tau'} \right),
\end{equation}
where $\sigma_t^\tau$ denotes the censored chain with initial condition $\sigma_0^\tau = \tau$. Note that the inequality holds for any coupling $(\sigma_t^\tau, \sigma_t^{\tau'} )$. Let $\tau, \tau' \in T_n$ be arbitrary and for each $a\in[n]$, define
\begin{equation*}
\widetilde{T}_a:= \min \{t \geq 0: (\sigma_{t}^\tau)^{-1}(a) = (\sigma_t^{\tau'})^{-1}(a) \}
\end{equation*}
to be the coupling time of the card $a$ in both decks. In order to estimate the decay of the coupling time between $ \sigma_t^{\tau}$ and $\sigma_t^{\tau'}$, we adopt the coupling  which differs from the monotone coupling in \S \ref{subsecmono}. This can be described as follows: \vspace{1mm}

  At time $(n-1)s$ for each $s =0,1,2, \ldots$, we choose the same orientation of exploration in both decks.
At time $t$, let $(x,x+1)$ denote  the edge that we are about to swap  or not.
\begin{enumerate}
	\item[($1$)] If $\sigma_t(x)= \sigma_t'(x+1)$ or $\sigma_t(x+1)= \sigma_t'(x)$ then we do opposite moves. In other words, we pick either $\sigma_t$ or $\sigma_t'$ uniformly at random and swap the cards at positions $x, x+1$ of the chosen one while leaving the other fixed.
	\item[($2$)] Otherwise, we do identical moves; we either transpose the cards at $x, x+1$ for both $\sigma_t$ and $\sigma'_t$ or do nothing for both of them, each with probability $1/2$.
\end{enumerate}
This rule ensures that once a specific card is in the same position in both decks, then it will remain matched forever. Thus, if $a\in (x_{i-1} ,x_i]$ and $\tau^{-1}(a) \leq (\tau')^{-1}(a)$, then $\widetilde{T}_a$ is bounded by the hitting time $T_a$ defined as
\begin{equation*}
T_a := \min \{t\geq 0 : (\sigma_t^\tau)^{-1}(a) = x_i  \}.
\end{equation*} 
Therefore, we are in the identical situation as Lemma \ref{couptimelem}, except that the length of the interval which the process $(\sigma_t^\tau)^{-1} (a)$ can move around is now $\Delta x_i \leq \lfloor \eta n \rfloor+1$. Therefore, Lemma \ref{couptimelem} gives that
\begin{equation*}
\mathbb{P}(T_a >t_1 ) \leq (1+ O(n^{-1})) \exp \left( -\frac{\pi^2}{5\eta^2 n^3} t_1 \right) ,
\end{equation*}

\noindent and by a union bound over all $a\in[n]$ we obtain that
\begin{equation*}
\mathbb{P}\left( \sigma_{t_1}^\tau \neq \sigma_{t_1}^{\tau'}\right) \leq
\sum_{a=1}^n \mathbb{P} (T_a >t_1 )
\leq n \exp \left(-\frac{\pi^2}{15\eta} \log n \right)\leq \varepsilon/3,
\end{equation*}
for all sufficiently large $n$. Combining with (\ref{cencoupeq}) implies the desired result. \qed

\begin{remark}
	\textnormal{In \cite{Lacoin} where they study the random AT shuffle, the censored shuffle under the same cencoring scheme $\mathcal{C}$ during time $[0,t_1]$ simplifies to the product chain of $K$ copies of independent random AT shuffle on $(x_{i-1},x_i]$. Thus, they prove Proposition \ref{t1prop} using this fact without introducing the above coupling.   }
\end{remark}

\vspace{3mm}

In order to prove Proposition \ref{t3prop}, we need the following proposition:

\begin{proposition}\label{t2prop}
	For any given $\eta, \varepsilon >0$, the following holds for all large enough $n$:
	\begin{equation*}
	|| \bar{\nu}_{t_2} - \bar{\mu} ||_{T.V. }\leq \varepsilon/3.
	\end{equation*} 
\end{proposition}

\noindent $\blacktriangleright$ \textit{Proof of Proposition \ref{t2prop}.} ~We define the function $h:S_n \rightarrow \mathbb{R}$ to be
\begin{equation*}
h(\sigma) := \sum_{i,j=1}^{K-1} \bar{\sigma}(i,j).
\end{equation*}
Then for any increasing probability measure $\nu$ on $S_n$, we have the following lemma from \cite{Lacoin} which tells us how the expected value $\nu(h)$ controls the distance $||\nu - \mu||_{T.V.}$ from the uniform measure:

\begin{lemma}[\cite{Lacoin}, Lemma 5.5]\label{smallmeanlem}
	Let $\nu$ be an increasing probability measure on $S_n$. For all $\varepsilon >0$, there exists a constant $\gamma(K,\varepsilon)>0$ such that for all sufficiently large $n$,
	\begin{equation*}
	\nu(h) \leq \gamma \sqrt{n} ~~~~\textnormal{implies}~~~~ ||\bar{\nu} - \bar{\mu}||_{T.V.} \leq \varepsilon/3.
	\end{equation*}
\end{lemma}

\noindent Lemma \ref{smallmeanlem} stems from the observation that if $\sigma \sim \mu$, then $n^{-1/2}\,\bar{\sigma}(i,j) $ converges to a Gaussian distribution as $n$ tends to infinity. Due to this fact, one can show that if $\nu(\bar{\sigma}(i,j))$ is less than a small constant times $\sqrt{n}$, then the distance between $\bar{\mu}_{i,j}$ and $\bar{\nu}_{i,j}$ is accordingly small, where $\bar\nu_{i,j}$ (resp. $\bar\mu_{i,j}$) denotes the distribution of $\bar{\sigma}(i,j)$ under $\sigma \sim \nu$ (resp. $\sigma \sim \mu$). The function $h$ combines the information for all $i, j$.

Due to Lemma \ref{tildebdlem}, $\nu_{t_2}(h)$ can be bounded by $\gamma \sqrt{n}$, and hence we can apply Lemma \ref{smallmeanlem} to obtain the desired inequality. Letting $\delta = \eta/7$ in Lemma \ref{tildebdlem}, we have
\begin{equation*}
\nu_{t_2}(h) \leq P^{t_2-t_1}_{\id} (h) \leq n(K-1)^2 \exp\left( -\left(1-\frac{\eta}{6}\right) \frac{\pi^2}{n^3} (t_2- t_1) \right) \leq
\gamma \sqrt{n},
\end{equation*}
where the last inequality holds for any fixed $\gamma>0$ when $n$ is large enough. Moreover, since $\textbf{1}_{\id} $ is increasing, Lemma \ref{increasinglem} implies that $\nu_{t_2}$ is also increasing. Therefore, Lemma \ref{smallmeanlem} tells us that 
\begin{equation*}
||\bar{\nu}_{t_2} -\bar{\mu}||_{T.V.} \leq \varepsilon /3. 
\end{equation*}
\qed


Now we conclude the proof of Proposition \ref{t3prop}. We again rely on the ideas in the proof of Proposition \ref{t1prop} and then follow .

\vspace{2mm}
\noindent $\blacktriangleright$ \textit{Proof of Proposition \ref{t3prop}.} ~ Let $\sigma_{t}$ be the state of the censored CAT shuffle at time $t$. Due to our censoring scheme, we have
\begin{equation*}
\sigma_{t} (\{x_{i-1}+1, \ldots, x_i\}) = \sigma_{t_2} (\{x_{i-1}+1, \ldots, x_i\}) ~~~~~\textnormal{for all } i\in [K],~ t\in [t_2, t_3].
\end{equation*}
Therefore, the stationary distribution $\mu_{\sigma_{t_2}}$ for the chain during time $t\in[t_2, t_3] $ can be written as
\begin{equation*}
\mu_{\sigma_{t_2}}(\cdot) := \mu(\;\cdot\;|\,\sigma(\{x_{i-1}+1,\ldots,x_i\}) = \sigma_{t_2}(\{x_{i-1}+1, \ldots,x_i\} ), ~\forall i\in[K] ).
\end{equation*}
(Note the difference between $\mu_{\sigma_{t_2}} $ and $\textbf{1}_{\sigma_{t_2}}^u$; the former uniformizes over the positions $x\in (x_{i-1}, x_i]$ while the latter uniformizes over the labels $\sigma(x) \in (x_{i-1}, x_i]$.) Thus, the same coupling argument in Proposition \ref{t1prop} implies that
\begin{equation}\label{t3bd1}
||\nu_{t_3}( \; \cdot \; | \sigma_{t_2}) - \mu_{\sigma_{t_2}}||_{T.V.} \leq \varepsilon/3,
\end{equation}
where $\nu_{t_3}( \; \cdot \; | \sigma_{t_2})$ denotes the probability distribution of $\sigma_{t_3}$ given that it was at state $\sigma_{t_2}$ at time $t_2$. For arbitrary $\xi \in \bar{S}_n$,  we average the inequality (\ref{t3bd1}) on the event $\{\bar{\sigma}= \xi \}$ to obtain that
\begin{equation*}
\sum_{\sigma_{t_2} : \, \bar{\sigma}_{t_2  } = \xi} \nu_{t_2}(\sigma_{t_2} |\, \bar{\sigma}_{t_2} = \xi) \, ||\nu_{t_3}( \; \cdot \; | \sigma_{t_2}) - \mu_{\sigma_{t_2}}||_{T.V.}\; \geq\; 
|| {\nu}_{t_3} (\;\cdot\; |\, \bar{\sigma}_{t_2}=\xi) - {\mu} (\;\cdot\; |\, \bar{\sigma}= \xi)||_{T.V.} .
\end{equation*}
Thus, by taking projections and using (\ref{t3bd1}) we have
\begin{equation}\label{t3bd2}
|| \widehat{\nu}_{t_3} (\;\cdot\; |\, \bar{\sigma}_{t_2}=\xi) - \widehat{\mu} (\;\cdot\; |\, \bar{\sigma}= \xi)||_{T.V.} \leq \varepsilon/3.
\end{equation}


\noindent Following the computation in Proposition 5.3 of \cite{Lacoin}, this implies that
\begin{equation*}
\begin{split}
\sum_{\widehat\sigma \in \widehat{S}_n}|\widehat{\nu}_{t_3} (\widehat{\sigma} ) -
\widehat{\mu}(\widehat{\sigma})| 
&\leq
\sum_{\xi \in \bar{S}_n}\; \sum_{\widehat{\sigma}:\; \bar{\sigma}=\xi}  |\widehat{\nu}_{t_3} (\widehat{\sigma} ) -
\widehat{\mu}(\widehat{\sigma})| \\
&\leq
\sum_{\xi \in \bar{S}_n}\; \sum_{\widehat{\sigma}:\; \bar{\sigma}=\xi}  \Big( \, \bar{\nu}_{t_2}(\xi) \;  |\widehat{\nu}_{t_3} (\widehat{\sigma}\,|\, \bar{\sigma}_{t_2}=\xi ) -
\widehat{\mu}(\widehat{\sigma} \,|\, \bar{\sigma}=\xi )|  \\
&~~~~~~~~~~~~~~~~~~+ \bar{\mu}(\widehat{\sigma} \,|\, \widehat{\sigma}=\xi ) \; | \bar{\nu}_{t_2} (\xi) - \bar{\mu}(\xi) | \,\Big) \\
&\leq
\frac{2\varepsilon}{3} +\frac{2\varepsilon}{3} \leq \frac{4\varepsilon}{3},
\end{split}
\end{equation*}
where the inequality in the last line follows from (\ref{t3bd2}) and Proposition \ref{t2prop}. \qed

\section{Application to the systematic simple exclusion process}\label{secex}

In this section, we study the systematic simple exclusion process using the techniques developed from the previous chapters. We show that the mixing time of this process satisfies a similar bound as Theorem \ref{main}.

The systematic simple exclusion process  can be understood as a projection of the CAT shuffle. To define the model, consider we have a length $(n-1)$ path on $\{1 , \ldots, n \}$ and locate $k \leq n$ particles at vertices, with each vertex being occupied by at most one particle. We introduce the dynamics similar to the CAT shuffle: At the beginning, we pick either $1$ or $n$ uniformly at random. If $1$ is chosen, then at time $t\in\{1,\ldots,n-1 \} $ we update the edge $(t,t+1)$, meaning that we either swap the possessions of the endpoints of the edge or leave it stay fixed, each with probability $\frac{1}{2}$. If $n$ is chosen, we explore in the opposite direction. After updating all $(n-1)$ edges, we again choose a random initial location out of  $\{1,n\}$ and continue the systematic updates starting from the chosen point. In other words, it is the projection of the CAT shuffle which regards $k$ cards  as particles and the rest as empty sites.

Using the argument from previous sections, we have the following mixing time bound for the systematic simple exclusion process.

\begin{theorem}\label{main2}
	Consider the systematic simple exclusion process on the line $\{1,\ldots,n \}$ with $k(n)$ particles such that both $k $ and $n-k$ tends to infinity as $n \rightarrow \infty$. Let $k' := \min \{k,n-k\}$. Then for any $\varepsilon>0$, we have
	\begin{enumerate}
		\item[\textnormal{(a)}] $t_{mix} (1-\varepsilon) 
		\geq \frac{n^3}{2\pi^2} \log k' 
		-\frac{n^3}{2\pi^2}\log \left(\frac{c\log k'}{ \varepsilon} \right) $, where  $c$ is a universal constant.
		
		\item[\textnormal{(b)}] $t_{mix} (\varepsilon) \leq (1+o(1)) \frac{4n^3}{\pi^2} \log k' $.
	\end{enumerate}
\end{theorem}

\begin{remark}
	\textnormal{We conjecture that the lower bound of Theorem \ref{main2} is sharp, i.e. the systematic simple exclusion process should exhibit cutoff at $t_{mix}(\varepsilon) = (1+o(1))\frac{n^3}{2\pi^2}\log k'$. The main difficulty of improving (b) of Theorem \ref{main2} stems from the deterministic aspects of the update rule. For instance, in  \cite{Lacoin} where cutoff for the simple exclusion process is established, the problem can be reduced to analyzing simple random walks. However in the systematic case, the increments of the random walks corresponding to those derived in \cite{Lacoin} are heavily correlated, which makes it more difficult to study. }
\end{remark}

\begin{proof}
	We can assume that $k \leq \frac{n}{2}$, since in the other case we can swap the roles of empty sites and particles. Let  $\Omega_{n,k} := \{\xi \in \{0,1\}^n: \sum_{x=1}^{n} \xi(x) =k \}$ be the state space of the chain, where $\xi(x) = 1$ (resp. $\xi(x)=0$) indicates that position $x$ is occupied (resp. empty).
	
	To prove part (a), We consider the following height function for each $\xi \in \Omega_{n,k}$:
	\begin{equation}\label{gdef}
	g_\xi (x): = \sum_{z=1}^x \xi(z) - \frac{xk}{n}.
	\end{equation}

	\noindent Using the height function, define
	\begin{equation}\label{Phiexdef}
	\Psi(\xi) := \sum_{x=1}^n g_\xi (x)\sin \left(\frac{\pi x}{n} \right).
	\end{equation}
	We additionally define $\wedge_t$ to be the state at time $t$ of the systematic simple exclusion process with initial condition that has  particles in the first $k$ positions (i.e., $\wedge(x) = \textbf{1}_{\{x\leq k \}}$ for all $x$), and let
	\begin{equation}\label{Psitdef}
	\Psi_t := \Psi (\wedge_{(n-1)t}).
	\end{equation}
	Then the following lemma is a straightforward generalization of Lemmas \ref{1stmlem} and \ref{2ndmlem}:
	
	\begin{lemma}\label{psims}
		Let $\Psi_t$ be defined as (\ref{Psitdef}). For any $t\in \mathbb{N}$ we have
		\begin{enumerate}
			\item[\textnormal{(a)}] $|\,\mathbb{E}[\Psi_{t+1}|\mathcal{F}_t ] -(1-\gamma)\Psi_t \,| \leq \frac{3\pi}{4n}$, where $\gamma := \pi^2/n^2 - O(n^{-4})$.
			
			\item[\textnormal{(b)}] $\mathbb{E}[(\Delta \Psi_t)^2|\mathcal{F}_t] \leq Ck\log k$, where $C>0$ is a universal constant.
		\end{enumerate}
		
	\end{lemma}
	
	\noindent By following the approach of \S \ref{subsecmain1pf}, we deduce part (a) from Lemmas \ref{W} and \ref{psims}. \vspace{2mm}
	
	To prove the upper bound, we consider two copies $\xi^1_t,\; \xi^2_t$ of systematic simple exclusion processes with different initial configurations, and estimate their coupling time using Lemma \ref{couptimelem}. To be specific, we first label the $k$ particles arbitrarily in both chains, and consider the coupling introduced in the proof of Proposition \ref{t1prop}. For each $i$, the coupling time of the $i$-th particle in $\xi_t^1$ and $\xi_t^2$ is bounded by the hiting time of the left particle reaching at the right end of the deck. Therefore, if we call the latter quantity $T_i$, then for any $\varepsilon, \delta>0$  and $t=(1+\delta)\frac{4n^3}{\pi^2} \log k $, Lemma \ref{couptimelem} implies that
	\begin{equation*}
	\max_{\xi_0^1,\: \xi_0^2 \in \Omega_n} \mathbb{P}\left( \xi_t^1 \neq \xi_t^2 \right) 
	\leq
	\sum_{i=1}^k \mathbb{P} (T_i >t) 
	\leq 
	\varepsilon,
	\end{equation*}
	for all sufficiently large n.
\end{proof}

\section{Appendix}

\subsection{The decay estimate: proof of Lemma \ref{tildebdlem}}\label{subsectildebd}
\noindent \textbf{Lemma \ref{tildebdlem}.} \textit{Let $(\sigma_t)$ denote the CAT shuffle starts from an arbitrary initial state  and  let $\delta>0$ be arbitrary. Then there exist $N_\delta, \:\theta_\delta>0$ such that  for any $x,y \in [n]$, $n \geq N_\delta$ and  $t > \theta_\delta n^3 $ satisfying $t = O(n^4)$,  we have}
\begin{equation}\label{eqdecay}
\left|\,\mathbb{E} [\widetilde{\sigma}_t (x,y) ]\,\right| \leq n(1+ O(tn^{-5})) \exp \left(-(1-\delta)\frac{\pi^2}{n^3} t \right).
\end{equation} 

\begin{proof}
	Let $y\in [n]$ be given and let $t$ be of the form $t=(n-1)i$ for $i\in \mathbb{N}$. Assume that $y<n$ (otherwise we have nothing to prove) and set $\Delta := n-1$. We analyze the expected difference between $\widetilde{\sigma}$ at time $t+\Delta$ and $t$ given the information $\mathcal{F}_t$ until time $t$. Recall Definition \ref{defmonocoup}, where we defined the random variables $U_s, \:\textbf{c}_i$ and the update rules using them. Notice that between time $t$ and $t+\Delta$, $\widetilde{\sigma}_s(x,y)$ can only be changed when updating the edge $(x,x+1)$. Also, when update is performed at edge $(x,x+1) $ at time $s$,  it goes up by $1$ if $\sigma_s(x)>y\geq \sigma_s(x+1)$ and $U_s =1$, whereas it moves down by $1$ if $\sigma_s(x)\leq y <\sigma_s(x+1)$ and $U_s =0$. 
	
	Set $v_t(x):= \widetilde{\sigma}_t(x,y)$. We compute $\mathbb{E} [v_{t+\Delta}(x) - v_t (x)\; | \; \mathcal{F}_t]$ based on the above properties of $\widetilde{\sigma}$, by considering the cases $\textbf{c}_i=1$ and $\textbf{c}_i =n$ separately. If $2\leq x \leq n-2$, we have
	\begin{equation}\label{diffcomp1}
	\begin{split}
	\mathbb{E} [v_{t+\Delta}(x) - v_t (x)\; | \; \mathcal{F}_t, \: \textbf{c}_i =1]
	=
	&\sum_{k=0}^{x-2} \frac{1}{2^{k+2}} \left(\textbf{1}_{\{\sigma_t(x+1) \leq y < \sigma_t (x-k) \}} - \textbf{1}_{\{\sigma_t(x+1) > y \geq \sigma_t(x-k) \}}  \right) \\
	&~~~+
	\frac{1}{2^{x}} \left(\textbf{1}_{\{\sigma_t(x+1) \leq y < \sigma_t (1) \}} - \textbf{1}_{\{\sigma_t(x+1) > y \geq \sigma_t(1) \}}  \right);\\
	\mathbb{E} [v_{t+\Delta}(x) - v_t (x)\; | \; \mathcal{F}_t, \: \textbf{c}_i =n]
	=
	&\sum_{k=0}^{n-x-2} \frac{1}{2^{k+2}} \left(\textbf{1}_{\{\sigma_t(x+1+k) \leq y < \sigma_t (x) \}} - \textbf{1}_{\{\sigma_t(x+1+k) > y \geq \sigma_t(x) \}}  \right)\\
	&~~~+
	\frac{1}{2^{n-x}} \left(\textbf{1}_{\{\sigma_t(n) \leq y < \sigma_t (x) \}} - \textbf{1}_{\{\sigma_t(n) > y \geq \sigma_t(x) \}}  \right).
	\end{split}
	\end{equation} 
	Notice the following relation between the indicators:
	\begin{equation*}
	\begin{split}
	\textbf{1}_{\{\sigma(x_1) \leq y < \sigma (x_2) \}} - \textbf{1}_{\{\sigma(x_1) > y \geq \sigma(x_2) \}}  
	&=
	\textbf{1}_{\{\sigma(x_1) \leq y \}} - \textbf{1}_{\{\sigma(x_1), \; \sigma(x_2)\leq  y  \}}\\  
	&~~~~- \textbf{1}_{\{\sigma(x_2) \leq y \}} + \textbf{1}_{\{\sigma(x_1), \; \sigma(x_2)\leq y  \}}\\
	&=   \textbf{1}_{\{\sigma(x_1) \leq y \}}- \textbf{1}_{\{\sigma(x_2) \leq y \}} \\
	&=
	\widetilde{\sigma}(x_1,y) - \widetilde{\sigma}(x_1 -1, y) - \widetilde{\sigma}(x_2,y) + \widetilde{\sigma}(x_2-1,y),
	\end{split}
	\end{equation*}
	where we define $\widetilde{\sigma}(0,y):= 0$. This property implies that
	\begin{equation}\label{diffcomp11}
	\begin{split}
	\mathbb{E} [v_{t+\Delta}(x) - v_t (x)\; | \; \mathcal{F}_t]
	=&
	\sum_{k=0}^{x-2} \frac{1}{2^{k+3}} \left\{ v_t(x+1) - v_t(x) - v_t(x-k) + v_t(x-k-1)  \right\}\\
	&~~~+ \frac{1}{2^{x+1}} \left\{v_t(x+1) - v_t(x) - v_t(1)   \right\}\\
	+&\sum_{k=0}^{n-x-2} \frac{1}{2^{k+3}} \left\{v_t(x+1+k) - v_t(x+k) - v_t(x) + v_t(x-1) \right\}\\
	&~~~+ \frac{1}{2^{n-x+1}} \left\{ - v_t(n-1) - v_t(x) + v_t(x-1) \right\}
	\end{split}
	\end{equation}
	Letting $\bar{v}_t(x) = \mathbb{E}[v_t(x)]$, taking expectations on both sides of (\ref{diffcomp11}) and rearranging the terms in the r.h.s., we have that for each $2 \leq x \leq n-2$, 
	\begin{equation}\label{vbareq1}
	\begin{split}
	\bar{v}_{t+\Delta}(x)
	=
	\sum_{k=-1}^{n-x-2} \frac{\bar{v}_t(x+k)}{2^{k+3}}
	~ +~\sum_{k=-1}^{x-2} 
	\frac{\bar{v}_t (x-k)}{2^{k+3}}.
	\end{split}
	\end{equation}
	Similar calculations for $x=1$ and $x=n-1$ yield that
	\begin{equation}\label{vbareq2}
	\begin{split}
	\bar{v}_{t+\Delta}(1)
	&=~
	\frac{1}{8}\bar{v}_t(1) + \frac{1}{4} \bar{v}_t(2)
	+\sum_{k=2}^{n-2} \frac{\bar{v}_t(k) }{2^{k+2}};\\
	\bar{v}_{t+\Delta}(n-1)
	&=~
	\frac{1}{8}\bar{v}_t(n-1) + \frac{1}{4} \bar{v}_t(n-2)
	+\sum_{k=2}^{n-2} \frac{\bar{v}_t(n-k) }{2^{k+2}}.
	\end{split}
	\end{equation}
	
	Due to the monotonicity of $\widetilde{\sigma}(x,y)$ in terms of $\sigma$, it suffices to prove the desired inequality (\ref{eqdecay}) for the initial condition $\sigma_0=\id$ which is the maximal case. The minimal case with initial state $\sigma^-(z)=n+1-z$ is also included in the maximal case; the only differences are the sign and taking $\widetilde{\sigma}^-(\cdot,n-y)$ instead of $\widetilde{\sigma}(\cdot, y)$. 
	
	Thus, let us assume  that  $\sigma_0= \id$.
	In order to establish the main inequality (\ref{eqdecay}), we will introduce $u_s : [n] \rightarrow \mathbb{R}$ which satisfies $u_s(x)\geq \bar{v}_{\Delta s}(x)$ and the bound
	\begin{equation*}
	|| u_s||_\infty \leq n(1+O(sn^{-4})) \exp \left( -(1-\delta)\frac{\pi^2}{n^2}s \right).
	\end{equation*}
	
	Let $u_0 (x):= \bar{v}_0 (x)$, $u_s(0)=u_s(n)=0$ and define $u_{s+1}(x)$ to follow (\ref{vbareq1}) so that
	\begin{equation}\label{ueq1}
	\begin{split}
	u_{s+1}(x)
	=~ 
	\sum_{k=-1}^{n-x-1} \frac{u_s(x+k)}{2^{k+3}}
	~ +~\sum_{k=-1}^{x-1} 
	\frac{u_s (x-k)}{2^{k+3}},
	\end{split}
	\end{equation}
	for each $x\in [n]$ and $j \in \mathbb{N}$. Note the difference between $\bar{v}_s $ and $u_s$ as  $\bar{v}_{s+1}(x)$  satisfies (\ref{vbareq1}) only for $2\leq x \leq n-2$.
	Since $u_0=\bar{v}_0$ is positive and the coefficients in (\ref{ueq1}) are at least as large as those in (\ref{vbareq1}) and (\ref{vbareq2}), we have $u_s \geq \bar{v}_s$ for all $s$. 
	
	Furthermore, we define $d_s : [n] \rightarrow \mathbb{R} $ by
	\begin{equation}\label{ddef}
	d_s (x) := u_s(x)-u_s(x-1),
	\end{equation} 
	We analyze $d_s$ instead of $u_s$ since it has a  tractable initial condition. Indeed, note that $||d_0||_\infty \leq 1$, which is much smaller compared to $||u_0||_\infty \asymp n.$ Also, note the obvious inequality that 
	\begin{equation}\label{ineqlinfl1}
	||u_s||_\infty \leq ||d_s||_1.
	\end{equation}
	Based on (\ref{ueq1}), we compute the transition rule of $d_s$ as follows:
	\begin{equation}\label{deq1}
	\begin{split}
	d_{s+1}(x) 
	=
	\sum_{k=-1}^{n-x} \frac{d_s(x+k)}{2^{k+3}} 
	~+~
	\sum_{k=-1}^{x-1} \frac{d_s(x-k)}{2^{k+3}} .
	\end{split}
	\end{equation}
	Therefore one can observe that the equation (\ref{deq1}) is equivalent to the transition rule of the random walk on $\mathbb{Z}$ that has i.i.d. increments $X_j \sim X$ with
	\begin{equation*}
	\mathbb{P}(X=k)= \frac{1}{2^{|k|+3}} + \frac{1}{2^{-|k|+3}} \textbf{1}_{\{|k|\leq 1 \}},
	\end{equation*}
	and that dies out when reaching outside of $[n] $. 
	
	Let $ S_m^x := x+ \sum_{j=1}^m X_j$ be the symmetric random walk on $\mathbb{Z}$ with i.i.d. $X_j \sim X$ that starts at $x$, and let $ \widehat{\tau}_n^x: = \min \{m\geq 0: S_m^x \notin [n] \}$ be its first exit time from $[n]$. For each $l\in [n]$, let $d_s^{(l)}$ denote the vector such that $d_0^{(l)} = \textbf{1}_{\{l \}}$  and follows the transition rule (\ref{deq1}). Then we have
	\begin{equation}\label{sumofu}
	 \left\| d_s^{(l)} \right\|_1 =\sum_{x=1}^{n} d_s^{(l)} (x) ~=~  \mathbb{P}(\widehat{\tau}_n^l >s) 
	.
	\end{equation}
	Thus, our goal is to bound
	the probability $\mathbb{P}(\widehat{\tau}_n^l >s)$, which can be done similarly as Lemmas \ref{couptimelem} and  \ref{wilsoncouptimelem}. Since Var$(X)=2$, Donsker's theorem implies that for any $\delta>0$, there exists $N_\delta$ such that 
	\begin{equation*}
	\mathbb{P} (\widehat{\tau}_n^{zn} > \theta n^2 ) 
	~\leq~ \mathbb{P} (\widehat{\tau}_B^{z/\sqrt{2}} > \theta-\delta )
	\end{equation*}
	for all $n \geq N_\delta$, where $\widehat{\tau}_B^z$ is the first exit time from $[0, 1/\sqrt{2}]$ of the standard Brownian motion with initial position $z$. Notice that we already have computed the probability in the r.h.s. in Lemma \ref{wilsoncouptimelem} and Corollary \ref{couptimecor}. According to these results, we obtain that for any constant $\theta>0$,
	\begin{equation*}
	\mathbb{P}(\widehat{\tau}_B^z > \theta) \leq C\exp(-\pi^2 \theta),
	\end{equation*}
	for some absolute constant $C>0$. (Although Lemma \ref{wilsoncouptimelem} is proven for the walk that starts at the midpoint of the given interval, generalization to the arbitrary starting location is straightforward.) Repeating the argument done in Corollary \ref{couptimecor} and Lemma \ref{couptimelem}, we deduce the following: For any $\delta>0$, there exist $\theta_\delta>0$ and $N_\delta >0$ such that for all $\theta \geq \theta_\delta$, $n\geq N_\delta$ and $z\in [n]$,
	\begin{equation*}
	\mathbb{P}(\widehat{\tau}_n^z > \theta n^2) \leq (1+O(\theta n^{-2}))\exp(-(1-\delta)\pi^2 \theta).
	\end{equation*}
	
	\noindent The original vector $d_s$ can be written as
	\begin{equation*}
	d_s = \sum_{l=1}^n d_0(l)\cdot d_s^{(l)}.
	\end{equation*}
	Since $||d_0||_\infty \leq 1$, we have 
	\begin{equation*}
	\left\| d_s \right\|_1 \leq \sum_{l=1}^n |d_0(l)| 
	\left\| d_s^{(l)} \right\|_1 \leq n (1+O(s n^{-4})) \exp \left(-(1-\delta) \frac{\pi^2}{n^2} s \right),
	\end{equation*}
	for all large $n>N_\delta$ and $ \theta_\delta n^2 < s \leq n^3$. Therefore, we deduce the desired result by (\ref{ineqlinfl1}).
\end{proof}

\subsection{Proof of Lemma \ref{wilsoncouptimelem}}\label{subsecwilson}

\noindent\textbf{Lemma \ref{wilsoncouptimelem}.}\textit{
	Let $\bar{\tau}^n$ be the first time that the simple random walk on $\mathbb{Z}$ starting at the origin hits $\pm n$. There exists a constant $C>0$ that satisfies $\mathbb{P} (\bar{\tau}^n > \theta n^2) < C(1+ O(\theta n^{-2})) \exp (-\pi^2 \theta/8)$ for all $\theta >0$ ($\theta$ may depend on $n$). 
}

\begin{remark}
	\textnormal{Lemma \ref{wilsoncouptimelem} is originally stated in terms of the hitting time at $\pm n / \sqrt{2}$. Here we presented an equivalent statement regarding the hitting time at $\pm n$.}
\end{remark}

\begin{proof}
	Let $\widetilde{S}_m$ denote the simple random walk on $\mathbb{Z}$ that starts at the origin. By the definition of $\bar{\tau}^n$, it suffices to show the desired inequality for $\tau^n_+$, where $\tau^n_+$ is the first time that $S^+_m := |\widetilde{S}_m| $ hits $n$. 
	
	Let $(Z_m)$ be the random walk on $\{0,1, \dots , n-1\}$ that has the same jump rate as $S^+_m$ on $\{ 0,1,\ldots, n-2\}$, and that at $(n-1)$ jumps to $(n-2)$ with probability $1/2$ and stays fixed otherwise. Then, one can notice that $\mathbb{P}(\tau^n_+ > \theta n^2)$ is equal to the survival probability of $Z_{\theta n^2}$. We focus on computing the latter quantity.
	
	Denote the transition matrix of $(Z_m)$ by $M_n$ and note that the matrix $M_n$ is symmetric. For each $j=0,1,\ldots,n-1$, let $f_j$ be the $n$-dimensional vector defined by 
	\begin{equation*}
	f_j (x) = \cos \left(\frac{(2j+1)\pi x}{2n} \right), ~~~\textnormal{for all } x\in \{0,1,\ldots,n-1\}.
	\end{equation*} 
	Observe that $f_j$'s are the eigenvectors of $M_n$, particularly since $\cos(\frac{(2j+1)\pi x}{2n} ) $ becomes zero at $x=n$. The corresponding eigenvalues are given by 
	\begin{equation*}
	\lambda_j = \cos \left(\frac{(2j+1)\pi}{2n} \right), ~~~\textnormal{for all } j\in \{0,1,\ldots, n-1 \}.
	\end{equation*}
	It is also straightforward to check that $f_j$'s are orthogonal:
	\begin{equation*}
	\sum_{x=0}^{n-1} f_j(x)f_k(x) = \frac{1}{2} \sum_{x=0}^{n-1} \cos \left(\frac{(j+k+1)\pi x}{n} \right) + \cos \left(\frac{(j-k)\pi x}{n} \right),
	\end{equation*}
	and the r.h.s. is nonzero if and only if $j=k$. Therefore, $\{f_j\}$ forms an orthogonal basis of the space of $n$-dimensional vectors. Let $\delta_0$ be the point mass at the origin. Elementary calculation yields that 
	\begin{equation}\label{tauplusineq}
	\begin{split}
	\mathbb{P}(\tau^n_+ > t) &= \sum_{x=0}^{n-1} M_n^t \delta_0 (x)  
	=\sum_{x=0}^{n-1} \sum_{j=0}^{n-1} \frac{\delta_0 \cdot f_j}{f_j \cdot f_j} \lambda_j^t f_j (x)\\
	&=\sum_{j=0}^{n-1} \frac{2}{n+1} \cos^t \left(\frac{(2j+1)\pi}{2n} \right) \sum_{x=0}^{n-1} \cos \left(\frac{(2j+1)\pi x}{2n} \right)
	~\leq~
	4\sum_{j=0}^{\lfloor n/2 \rfloor}   \cos^t \left(\frac{(2j+1)\pi}{2n} \right),
	\end{split}
	\end{equation}
	where in the second line we used the identity $f_j \cdot f_j = \frac{n+1}{2}$. If we consider the line passing $(0, \cos 0)$ and $(\alpha, \cos \alpha )$ for $\alpha = \pi/2n$, it lies above $(z, \cos z )$ for $z \in [\pi/2n, \pi/2]$. Thus, we can bound $\lambda_j^t$ by
	\begin{equation*}
	\cos^t \left(\frac{(2j+1)\pi}{2n} \right) \leq \left\{1- (2j+1) \left(1- \cos\left(\frac{\pi}{2n}\right)\right) \right\}^t ~\leq~ \exp \left\{- t(2j+1) \left(1- \cos\left(\frac{\pi}{2n}\right)\right)  \right\}.
	\end{equation*}
	Hence, summation over $j=0,\ldots, n-1$ yields that
	\begin{equation*}
	\sum_{j=0}^{\lfloor n/2 \rfloor}   \cos^t \left(\frac{(2j+1)\pi}{2n} \right)
	\leq 
	\frac{\exp \left(-\left(1- \cos \left( \frac{\pi}{2n}\right) \right)t \right) }{ 1-\exp \left(-\left(1- \cos \left( \frac{\pi}{2n}\right) \right)t \right) }
	~\leq~
	\left(1+ O\left(\frac{t}{n^4} \right) \right) \frac{\exp(-\pi^2 t/ 8n^2) }{1-\exp(-\pi^2 t/ 8n^2)}.
	\end{equation*}
	Therefore, combining with (\ref{tauplusineq}), we obtain that
	\begin{equation*}
	\mathbb{P}(\tau_+^n > \theta n^2) \leq \min \left\{ 1, ~(4+ O(\theta n^{-2})) \frac{\exp(-\pi^2 \theta / 8)}{1- \exp(-\pi^2 \theta/8)} \right\} 
	\leq (5+ O(\theta n^{-2})) \exp \left( - \frac{\pi^2}{8}\theta \right),
	\end{equation*}
	which is the desired result with $C=5$.
\end{proof}

\subsection{Proof of Lemma \ref{1stmlem} }\label{subsec1stm}

\noindent \textbf{Lemma \ref{1stmlem}.} 	\textit{Let $\Phi_t$ defined as (\ref{Phidef}). For any $t \in \mathbb{N}$ we have}
\begin{equation*}
|\, \mathbb{E} [\Phi_{t+1} | \mathcal{F}_t] - (1-\gamma) \Phi_t\,| \leq  \frac{4\pi}{3n} ,
\end{equation*}
\textit{where $\gamma := \pi^2/n^2 - O(n^{-4}).$}

\begin{proof} 
	According to the computations in (\ref{vbareq1}, \ref{vbareq2}), we have
	\begin{equation}\label{1stmeq22}
	\begin{split}
	\mathbb{E} [\Phi_{t+1}|\mathcal{F}_t] &= \sum_{x=2}^{n-2} \left[\sum_{k=-1}^{n-x-2} \frac{ h_t (x+k)}{2^{k+3}} + \sum_{k=-1}^{x-2} \frac{h_t(x-k)}{2^{k+3}} \right] \,\sin \frac{\pi x}{n}\\
	&+ \left[ \left( \sum_{k=0}^{n-3} \frac{h_t(1+k)}{2^{k+3}} \right) + \frac{1}{4} h_t(2)\right] \sin \left(\frac{\pi}{n} \right)\\
	&+  \left[ \left( \sum_{k=0}^{n-3} \frac{h_t(n-1-k)}{2^{k+3}} \right) + \frac{1}{4} h_t(n-2)\right] \sin \left(\frac{\pi(n-1)}{n} \right).
	\end{split}
	\end{equation}
	By rearranging the r.h.s., we obtain that
	\begin{equation}\label{1stmeq2}
	\begin{split}
	\mathbb{E} [\Phi_{t+1}|\mathcal{F}_t] 
	&=\sum_{y=2}^{n-2} \left[\sum_{k=-1}^{n-y-1} \frac{h_t(y)}{2^{k+3}} \sin \left(\frac{\pi (y+k)}{n} \right) +
	\sum_{k=-1}^{y-1} \frac{h_t(y)}{2^{k+3}} \sin \left(\frac{\pi (y-k)}{n} \right) \right]\\
	&~~~~+ (h(1) + h(n-1)) \left(\frac{1}{4} \sin\left(\frac{2\pi}{n} \right) + \frac{1}{8} \sin\left(\frac{\pi}{n} \right) \right)\\
	&=\left[\sum_{k=-1}^\infty \frac{ \cos ( \frac{\pi k}{n})}{2^{k+2}} \right]\, \Phi_t
	~+~
	\sum_{y=2}^{n-2} h(y) \left\{\sum_{k=1}^\infty \frac{\sin (\frac{\pi k}{n})}{2^{n-y+3+k}} +  \sum_{k=1}^\infty \frac{\sin (\frac{\pi k}{n})}{2^{y+3+k}} \right\}\\
	&~~~~- 
	(h(1)+h(n-1)) \left[\sum_{k=1}^\infty \frac{3 \sin(\frac{\pi k}{n})}{2^{k+4}} \right].
	\end{split}
	\end{equation}
	Noting that 
	\begin{equation*}
	\sum_{k=-1}^\infty \frac{\cos(\frac{\pi k}{n})}{2^{k+3}} = 1-\frac{\pi^2}{n^2} + O\left(\frac{1}{n^4} \right) = 1-\gamma
	\end{equation*}
	as well as that $|h_t(x)| \leq \frac{1}{2} x \wedge (n-x) $, we can deduce from (\ref{1stmeq2}) that
	\begin{equation}\label{1stmeq4}
	\begin{split}
	|\, \mathbb{E}[\Phi_{t+1} | \mathcal{F}_t] - (1-\gamma)\Phi_t \,| 
	~\leq~ \sum_{y=2}^{\infty} y \left\{ \sum_{k=1}^\infty \frac{\pi}{n} \frac{k}{2^{y+3+k}} \right\} 
	\,+\,\sum_{k=1}^\infty \frac{3k}{2^{k+4}} \frac{\pi}{n}
	~=~
	\frac{3\pi}{4n}.
	\end{split}
	\end{equation}
\end{proof}


\bibliographystyle{plain}
\bibliography{cat}

\end{document}